\documentclass[a4paper, 11pt]{article}
\usepackage{geometry}
\usepackage[dvipdfmx]{graphicx, hyperref}

\usepackage{amsmath}						
\usepackage{amsfonts}
\usepackage{latexsym}
\usepackage{amssymb}
\usepackage{bm}

\usepackage{cases}

\usepackage{wrapfig,float}
\usepackage{hyperref}

\renewcommand{\labelenumi}{(\roman{enumi})}
\allowdisplaybreaks[3]

\usepackage{amsthm}							
\usepackage{ascmac}
\usepackage{mathtools}

\usepackage[right]{lineno}       		
\theoremstyle{definition}
\newtheorem{theorem}{Theorem}[section]
\newtheorem{prop}[theorem]{Proposition}
\newtheorem{lemma}[theorem]{Lemma}
\newtheorem{cor}[theorem]{Corollary}
\newtheorem{rem}[theorem]{Remark}
\newtheorem{definition}[theorem]{Definition}
\newtheorem{example}[theorem]{Example}
\newtheorem{setting}[theorem]{Setting}
\newtheorem{notation}[theorem]{Notation}

\newcounter{res}

\newtheorem{result}[res]{Main Result}   				

\newcommand{\Leb}{\mathrm{Leb}}				

\DeclareMathOperator*{\rs}{\widehat{\mathbb{C}}}
\DeclareMathOperator*{\nn}{\mathbb{N}}

\DeclareMathOperator*{\poly}{\mathrm{Poly}}

\DeclareMathOperator{\supp}{supp}
\DeclareMathOperator{\diam}{diam}

\DeclareMathOperator*{\rat+}{\mathrm{Rat}_{+}}
\DeclareMathOperator*{\Poly}{\mathrm{Poly}}
\DeclareMathOperator*{\Cpt}{\mathrm{Cpt}}

\DeclareMathOperator*{\CC}{\mathbb{C}}
\DeclareMathOperator*{\hypd}{\mathit{d}_\mathrm{{hyp}}}

\newcommand{\MRDS}{\textrm{MRDS}}

\newcommand{\chaotic}{\textrm{chaotic}}
\newcommand{\condition}{\textrm{non-degenerate}}

\title{Non-i.i.d. random holomorphic dynamical systems\\ and the generic dichotomy
\footnote{ 
To appear in Nonlinearity. 
}}
\date{January 24, 2022}

\author{Hiroki Sumi\footnote{Course of Mathematical Science, Department of Human Coexistence,
 Graduate School of Human and Environmental Studies, Kyoto University.  
  Yoshida Nihonmatsu-cho, Sakyo-ku, Kyoto, 606-8501, Japan}\\
E-mail: sumi@math.h.kyoto-u.ac.jp\\
\href{http://www.math.h.kyoto-u.ac.jp/~sumi/index.html}{http://www.math.h.kyoto-u.ac.jp/$\sim$sumi/index.html}\\  \\
Takayuki Watanabe\footnote{Author to whom any correspondence should be addressed. Department of Mathematics, Faculty of Science, Kyoto University. Kitashirakawa Oiwake-cho, Sakyo-ku, Kyoto 606-8502, Japan}\\
E-mail: watanabe.takayuki.3m@kyoto-u.ac.jp}

\begin{document}

\maketitle
\begin{abstract}
We consider non-i.i.d.\ random holomorphic dynamical systems whose choice of maps depends on Markovian rules. 
We show that generically, such a system is mean stable or chaotic with full Julia set. 
If a system is mean stable, then the Lyapunov exponent is uniformly negative for every initial value and almost every random orbit.  
Moreover, we consider  families of random holomorphic dynamical systems 
and show that  the set of mean stable systems has full measure under certain conditions.  
The latter is a new result even for i.i.d.\ random dynamical systems. 
\end{abstract}

\section{Introduction}
\subsection{Background}
We consider random dynamical systems (RDSs) of rational maps on the Riemann sphere $\rs.$  
The study of RDS is rapidly growing. 
The previous works find many new phenomena which cannot happen in deterministic dynamics, which are called noise-induced phenomena or randomness-induced phenomena. 
For example, chaotic dynamics can be more chaotic if one adds noise, and 
more surprisingly, chaotic dynamics can be more stable because of noise. 
The latter phenomena are called noise-induced order. 
For details on randomness-induced phenomena, the reader is referred to the authors' previous paper \cite{sw19} and references therein, say \cite{bb03, fs91, js17, msu, sumi13,  sumi21, uz16}.

However, most of the previous studies concerned i.i.d.\ random dynamical systems. 
It is very natural to generalize the settings and consider non-i.i.d.\ random dynamical systems. 
In this paper, we especially treat random dynamical systems with “Markovian rules” whose randomness depends on the past.

Our studies may be applied to the skew products whose base dynamical systems have Markov partitions. 
We believe that this research will contribute not only toward mathematics but also toward applications to the real world. 
One motivation for studying dynamical systems is to analyze mathematical models used in the natural or social sciences. 
Since the environment changes randomly, it is natural to investigate random dynamical systems which describe the time evolution of systems with probabilistic terms. 
In this sense, it is very important to understand “Markovian” noise since there are a lot of systems whose noise depends on the past.

The authors found a noise-induced phenomenon which can happen in Markov RDSs but cannot happen in i.i.d.\ RDSs, see \cite[Main Result 6]{sw19}. 
This exhibits the difference between non-i.i.d.\ and i.i.d.\ RDSs and motivates us to study non-i.i.d.\ RDSs.

In this paper, we show some results regarding noise-induced order which greatly deepen the results in  \cite{sw19}. 
RDS with Markovian noise is the theme of this paper. 
In \cite{sw19},  the authors introduced Markov random dynamical systems quite generally but 
in this paper, we are concerned with rational maps and define such systems as follows.   
The properties of holomorphic functions allow us to control minimal sets and to study global (random) dynamics. 
For example, we use Montel's Theorem and hyperbolic metric  to show our results. 
Let $\rat+$ be the space of all rational maps of degree two or more from $\rs$ to itself endowed with the metric $\kappa ( f , g) := \sup_{z \in \rs} d (f(z),  g(z)  ),$ 
where   $d$ denotes the spherical distance on $\rs$.

\begin{definition}\label{def:tauIntro}
Let $m \in \nn$.
Suppose that $m^2$ regular Borel measures $ (\tau_{ij})_{i, j =1, \dots, m}$ on $\rat+$ satisfy  $ \sum_{ j =1}^m \tau_{i j}(\rat+) = 1$ for all $i =1, \dots, m$.
We call $\tau =  (\tau_{ij})_{i, j =1, \dots, m}$ a {\it Markov random dynamical system} (MRDS for short).
We say that $\tau$ is {\it compactly generated} if  $\supp \tau_{ij}$ is compact for each $i, j =1, \dots, m$. 
\end{definition}

For a given MRDS $\tau = (\tau_{ij})_{i, j =1, \dots, m}$, we consider the Markov chain on $\rs \times \{1, \dots, m\}$ whose transition probability from $(z, i ) \in \rs \times \{1, \dots, m\}$ to $ B \times \{ j \}$ is  defined by 
$$ \tau_{ij}(\{ f \in \rat+ ; f(z) \in B\}) ,$$
 where $B$ is a Borel subset of $\rs$  and $ j \in  \{1, \dots, m\}.$ 
 This (time-homogeneous) transition function defines the one-point motion on $\rs \times \{1, \dots, m\}$. 
 We can  construct skew-product maps also as a representation of Markov RDSs, 
 see Definition 2.29 of \cite{sw19} and 2.1.6 Theorem (RDS Corresponding to Markov Chain) of Arnold's book \cite{Arn}. 
 For general relation between Markov chain and random mappings, see pp.\ 53--55 of  \cite{Arn}.

The Markov chain induced by $\tau = (\tau_{ij})_{i, j =1, \dots, m}$ describes the following random dynamical system on the phase space $\rs.$
Fix an initial point $z_{0} \in \rs$ and choose a vertex $i = 1, \dots, m$ (with some probability if we like). 
We choose a vertex $i_{1} = 1, \dots, m$ with probability $\tau_{i i_{1} }(\rat+) > 0$ and choose a map $f_{1}$ according to the probability distribution $ \tau_{i i_{1}} / \tau_{i i_{1} }(\rat+).$ 
Repeating this, we randomly choose a vertex $i_{n}$ and a map $f_{n}$ for each $n$-th step. 
In this paper, we investigate the behavior of random orbits of the form $f_{n} \circ \dots \circ f_{2} \circ f_{1} (z_{0}).$ 
For the general theory of RDSs, see Arnold's book \cite{Arn}.

By extending the phase space from $\rs$ to $\rs \times \{ 1, \dots, m \}$, we can represent MRDSs as naive Markov chains. 
This simple representation enables us  to analyze MRDSs intuitively. 

Note that our definition is a generalization of i.i.d.\ RDS and deterministic dynamics. 
If $m=1$ then our definition coincides with i.i.d.\ RDS on $\rs$ induced by $\tau = \tau_{11}$. 
Besides, if $\tau_{11}$ is the Dirac measure $\delta_{f}$ at $f \in  \rat+$, then our definition treats  dynamics of iteration of $f$ essentially. 

\subsection{Definitions}
For an MRDS $\tau$, we define the following set-valued dynamics. 
We present our results in the next subsection \ref{ssec:MainRes}. 

\begin{definition}\label{def:graph}
Let $\tau =  (\tau_{ij})_{i, j =1, \dots, m}$ be an MRDS. 
We consider the directed graph $(V, E)$ in  the following way. 
We define the vertex set as $V := \{ 1,2, \dots , m\} $ and the edge set as 
$$E :=\{ (i ,j) \in V  \times V ; \, \tau _{ij}(\rat+)  > 0 \}. $$
Define $i : E \to V$ (resp. $t : E \to V$) as the projection to the  first (resp. second) coordinate and we call  $i(e)$ (resp. $t(e)$) the initial (resp. terminal) vertex of $e \in E$. 
We call $(V, E)$ the associated directed graph of $\tau$. 
Also, for each $e=(i,j) \in E$, we define $\Gamma_e :=  \supp \tau _{ij} $. 
Set $S_{\tau} :=(V, E, (\Gamma_e)_{e \in E})$, which we call the \textit{graph directed Markov system} (GDMS for short) induced by $\tau$. 
We say that $\tau$ is \textit{irreducible} if the associated directed graph $(V, E)$ is (strongly) connected. 
\end{definition}

Although one may think that our concept is similar to that of \cite{mu03}, ours is completely different from \cite{mu03}. 
Mauldin and Urba\'nski are concerned with the limit sets of systems of contracting maps, 
but in this paper, we discuss the dynamics and the Julia sets of GDMS consisting of rational maps 
which may have expanding property somewhere in the phase space.  

We denote by $\Poly$ the set of all polynomial maps of degree two or more. 
We work on subfamilies of $\rat+$ which satisfy the following $\condition$ condition. 

\begin{definition}\label{def:condition}
We say that a non-empty subset ${X}$ of $\rat+$ is {\condition}  if there exist an open subset $A$ of $\rat+$ and a closed subset  $B$ of $\rat+$ such that ${X}= A \cap B$  and at least one of the following (i) and (ii) holds. 
\begin{enumerate}
\item For each $(f_{0}, z_{0}) \in X\times \rs $, there exists a holomorphic family $\{ g_{\theta} \}_{\theta\in \Theta}$ of rational maps parametrized by a finite dimensional complex manifold $\Theta$ 
(i.e., $g_{\theta} \in \rat+$ for each $\theta \in \Theta$ and $(z, \theta) \mapsto g_{\theta}(z)$ is a holomorphic map from $\rs \times \Theta$ to $\rs$)
with $\{ g_{\theta} ; {\theta \in \Theta}\} \subset X$  such that $ g_{\theta_{0}} = f_{0}$ for some $\theta_{0} \in \Theta$ and $\theta \mapsto g_{\theta}( z_{0}) $ is non-constant in any neighborhood of $\theta_{0}$.

\item $X \subset \Poly$  and for each $(f_{0}, z_{0}) \in X\times \mathbb{C} $, there exists a holomorphic family $\{ g_{\theta} \}_{\theta \in\Theta}$ of rational maps parametrized by a complex manifold $\Theta$ with $\{ g_{\theta}; {\theta \in \Theta}\}\subset X$
 such that $ g_{\theta_{0}} = f_{0}$ for some $\theta_{0} \in\Theta$ and $\theta \mapsto g_{\theta}( z_{0}) $ is non-constant in any neighborhood of $\theta_{0}$. 
\end{enumerate}
\end{definition}

\begin{definition}\label{def:topMRDS}
Let  $X \subset \rat+$. 
Define $\MRDS(X)$ as the space of all  irreducible Markov random dynamical systems $\tau$ such that 
the topological  support $\supp \tau_{e}$ is compact and contained in $X$ for each $e \in E$, where $E$ is the set of directed edges of the associated directed graph of $\tau$.   

We endow $\MRDS(X)$ with  the following topology. 
A sequence $\{ \tau^{n} \}_{n \in \nn}$ in $\MRDS(X)$ converges to $\tau  \in \MRDS(X)$ if and only if 
\begin{enumerate}
\item the associated directed graph of $\tau^{n}$ is equal to  $(V, E)$ for sufficiently large $n$, where  $(V, E)$ denotes the associated directed graph of $\tau$,
\item the sequence of compact sets $\{ \supp \tau^{n}_{e} \}_{n\in \nn}$ converges to $\supp \tau_{e}$ with respect to  the Hausdorff metric for each directed edge $e \in E$, and 
\item the sequence of measures $\{ \tau^{n}_{e} \}_{n \in \nn}$ converges to $\tau_{e}$ for each $e \in E$  in the weak*-topology. 
\end{enumerate}
\end{definition}

\begin{definition}\label{def:FJ}
Let $S_\tau= (V, E, (\Gamma_e)_{e \in E})$ be the GDMS induced by an MRDS $\tau$.
\begin{enumerate}
\item A word  $e = (e_1,e_2, \dots, e_N) \in E^N$ with length $N \in \nn $ is said to be {\it admissible} if $t(e_n) = i(e_{n+1})$  for all  $n =1,2, \dots , N-1$. 
For this word $e$, we call $i(e_1)$ (resp. $t(e_N)$) the initial (resp. terminal) vertex of $e$ and we denote it by $i(e)$ (resp. $t(e)$). 
\item We set 
\begin{align*}
H(S_{\tau})&:= \{f_{N} \circ \dots  \circ f_{2} \circ f_{1} ;\\
& N \in \nn, f_{n} \in \Gamma_{e_n}, t(e_n)=i(e_{n+1})(\forall n=1,\ldots, N-1)\}, \\
H_i(S_{\tau})&:= \{f_{N} \circ  \dots \circ f_{2} \circ f_{1} \in H(S_{\tau}) ;\\
& N \in \nn, f_{n} \in \Gamma_{e_n}, t(e_n)=i(e_{n+1})(\forall n=1,\ldots, N-1),  i=i(e_1)\}, \\
H_i^j(S_\tau) &:= \{f_{N} \circ  \dots \circ f_{2} \circ f_{1} \in H_{i}(S_{\tau}) ;\\
& N \in \nn, f_{n} \in \Gamma_{e_n}, t(e_n)=i(e_{n+1})(\forall n=1,\ldots, N-1),  i=i(e_1), t(e_N)=j\}. 
\end{align*}
\item For each $i \in V$, we denote by $F_i(S_\tau)$  the set of  all points $z \in \rs$ for which there exists a neighborhood $U$ of $z$ in $\rs$ such that  the family  $\bigcup_{j\in V}H_i^j(S_\tau)$ of maps on $\rs$ is  equicontinuous on $U$. 
The set $F_i(S_\tau)$ is called the \textit{Fatou set} of $S_\tau$ at the vertex $i$,   and the complement $J_i(S_\tau) := \rs \setminus F_i(S_\tau)$ is called the \textit{Julia set} of $S_\tau$ at the vertex $i$.
\item Set $\mathbb{F}(S_\tau) :=\bigcup_{i \in V} F_i(S_\tau) \times \{i\}$ and $\mathbb{J}(S_\tau) :=\bigcup_{i \in V} J_i(S_\tau) \times \{i\}$.
\end{enumerate}
\end{definition}

We refer the readers to \cite{sw19} for examples of Julia sets of GDMSs and MRDSs. 

\subsection{Main Results}\label{ssec:MainRes}
In this subsection, we present the main results (Main Results \ref{mr:contractingBall}-\ref{mr:biffam}) of this paper. 
We first consider mean stable systems.

\begin{definition}\label{def:meanstable}
Let $\tau \in \MRDS(\rat+)$. We say that $\tau $ is \textit{mean stable} if the associated GDMS $S_{\tau}$ satisfies the following. 
There exist $N \in \nn$ and  two families of non-empty open sets $(U_{i})_{i \in V}$ and $(W_{i})_{i \in V}$ such that 
\begin{enumerate}
\renewcommand{\labelenumi}{(\Roman{enumi})}
\item $U_{i} \subset \overline{U_{i}} \subset W_{i} \subset \overline{W_{i}} \subset F_{i}(S_\tau)$ for each $i \in V,$ 
\item for each admissible word $e=(e_{1}, \dots, e_{N})$ with length $N$ and each $f_{n} \in \Gamma_{e_{n}} \, (n=1,\dots, N)$, we have $\overline{f_{N} \circ \dots  \circ f_{1} (W_{i(e)})} \subset U_{t(e)},$ and   
\item for each $z \in \rs$ and $i \in V$, there exist $j \in V_{}$ and $h \in H_{i}^{j}(S_\tau)$ such that $h(z) \in W_{j}.$ 
\end{enumerate}
\end{definition}

Forn\ae ss and Sibony proved in \cite{fs91} that small ``random perturbations'' of iteration of a single map give examples of mean stable systems, 
although they do not use the concept of mean stability. 
In this paper, we show that there exist a lot of mean stable systems in the space of MRDS. 
See also Example \ref{ex:MS}.

We show that if $\tau$ is mean stable, then  for every initial value, the sample-wise dynamics is contractive and the Lyapunov exponent is negative almost surely with respect to the probability measure $\tilde{\tau}$ on $(\rat+ \times E)^{\nn}$ (Main Results \ref{mr:contractingBall} and \ref{mr:NegLyap}). 
Here, $\tilde{\tau}$ is the measure naturally  associated with the Markov chain on $\rs \times \{1, \dots, m\}$   induced by $\tau$. 
See Definition \ref{titil_def} or \cite[Lemma 3.4]{sw19}. 

\begin{rem}
If an MRDS $\tau $ is mean stable, then the kernel Julia set is empty. 
In other words,  for every $i \in V$ and for every $z \in \rs$, there exist $j \in V$ and $h \in H_i^j(S_\tau)$ such that $h(z) \notin J_{j}(S_{\tau}),$ see \cite{sw19}. 
This property implies some interesting results. 
For instance, almost every random Julia set has Lebesgue measure $0.$ 
That is,  there exists a Borel set $\mathfrak{G}$ with  $\tilde{\tau} (\mathfrak{G}) =1 $ such that for every $(f_{n}, e_{n})_{n =1}^{\infty} \in \mathfrak{G}$, the Lebesgue measure of the complement of 
the set of  all points $z \in \rs$ for which there exists a neighborhood $U$ of $z$ in $\rs$ such that  the family 
$\{f_{n} \circ \dots  \circ f_{1}\}_{n=1}^{\infty}$ is  equicontinuous on $U$ 
is $0$. 
See \cite[Proposition 3.11]{sw19}. 

Additionally, for each attractor (attracting minimal set) $A$, 
the probability of random orbits tending to $A$ depends continuously on initial points. 
See \cite[Proposition 4.24]{sw19} and \cite{sumi11}. 
For the definition of attracting minimal set, see Definition \ref{def:attractingMinimal}. 
\end{rem}

\begin{result}[Theorem \ref{th:ContraBall}]\label{mr:contractingBall}
Let ${\tau} \in \MRDS(\rat+)$ be a mean stable system. 
Then we have all of the following. 
\begin{enumerate}
\item  There exists a constant $c \in (0,\, 1)$ satisfying that for each $z  \in \rs$, there exists a Borel subset $\mathfrak{F}$ of $(\rat+ \times E)^{\nn}$ with  $\tilde{\tau} (\mathfrak{F}) =1 $ 
 such that for every $(f_{n}, e_{n})_{n =1}^{\infty} \in \mathfrak{F}$, there exist $r= r(z, (f_{n}, e_{n})_{n =1}^{\infty}) >0$ and $K=K(z, (f_{n}, e_{n})_{n =1}^{\infty}) >0$ such that  $$ \diam  f_{n} \circ \dots  \circ f_{1} (B(z, r)) \leq K c^{n}$$ for every $n \in\nn.$   
 Here, we set $\mathrm{diam} A = \sup_{x, y \in A}  d(x, y)$ for any $A \subset \rs.$
\item  For each $z  \in \rs$, there exists a Borel subset $\mathfrak{F}$ of $(\rat+ \times E)^{\nn}$ with  $\tilde{\tau} (\mathfrak{F}) =1 $ 
 such that for every $(f_{n}, e_{n})_{n =1}^{\infty} \in \mathfrak{F}$, there exists an attracting minimal set $(L_{i})_{i \in V}$  such that  $d(f_{n}\circ \dots \circ f_1 (z), L_{t(e_{n})}) \to 0$ as $n \to \infty$. 
\item For every $i \in V$, the kernel Julia set 
$J_{\ker,i} (S_\tau) :=  \bigcap_{ j \in V} \bigcap_{h \in H_i^j(S) } h^{-1} (J_j(S))$ at $i$ is empty. 
\item There exists a Borel subset $\mathfrak{G}$ of $(\rat+ \times E)^{\nn}$ with  $\tilde{\tau} (\mathfrak{G}) =1 $ such that for every 
$\xi = (f_{n}, e_{n})_{n =1}^{\infty} \in \mathfrak{G}$, 
the ($2$-dimensional) Lebesgue measure of the Julia set $J_{\xi}$ is zero.   
Here, the Julia set $J_{\xi}$  of $\xi$  is the complement of  the set of  all points $z \in \rs$ for which there exists a neighborhood $U$ of $z$ in $\rs$ such that  the family $\{f_{n} \circ \dots  \circ f_{1}\}_{n=1}^{\infty}$ is  equicontinuous on $U$.  
\item There exist at most finitely many minimal sets of $S_{\tau}$. Moreover, 
each minimal set of $S_{\tau}$ is attracting. Also, 
for every minimal set $\mathbb{L} = (L_{i})_{i \in V}$ of $S_{\tau}$, 
we define the function $\mathbb{T}_{\mathbb{L}} \colon \rs \times V \to [0, 1]$ by
$$\mathbb{T}_{\mathbb{L}}(z,i) := \tilde{\tau}_i (\{ \xi = (f_n , e_n )_{n \in \nn}  ; \, d(f_{n}\circ \dots \circ f_1 (z), L_{t(e_{n})}) \to 0 \, (n \to \infty) \} )$$ 
for every point  $(z,i) \in \rs \times V$. 
Then $\mathbb{T}_{\mathbb{L}}$ is continuous on $\rs \times V$.  
Also, averaged function ${T}_{\mathbb{L}}(z) := \sum_{i \in V} p_{i} \mathbb{T}_{\mathbb{L}}(z,i)$ is continuous on $\rs,$ 
where  $(p_1,\dots,  p_m)$ is the probability vector in Definition \ref{titil_def}. 
\item Suppose $\tau$ has exactly $\ell$ minimal sets  $\mathbb{L}_{1}, \mathbb{L}_{2}, \dots, \mathbb{L}_{\ell}$. 
Then $${T}_{\mathbb{L}_{1}}(z) + {T}_{\mathbb{L}_{2}}(z) + \dots + {T}_{\mathbb{L}_{\ell}}(z) =1$$
for every $z \in \rs$. 
\end{enumerate}
\end{result}

\begin{result}[Theorem \ref{th:NegLyap}]\label{mr:NegLyap}
Let ${\tau} \in \MRDS(\rat+)$ be a mean stable system. 
Then there exists $\alpha < 0$ such that  the following holds. 
 For each $z  \in \rs$, 
there exists a Borel set $\mathfrak{F}$ with  $\tilde{\tau} (\mathfrak{F}) =1 $ such that for every $(f_{n}, e_{n})_{n =1}^{\infty} \in \mathfrak{F}$, we have 
$$\limsup_{n \to \infty} \frac{1}{n} \log \|D ( f_{n} \circ \dots  \circ f_{1}) (z) \| \leq \alpha.$$
Here, $Dg(z)$ denotes the complex differential of a holomorphic map $g$ at $z$ and $\| \cdot \| $ denotes the norm with respect to the spherical metric. 
\end{result}

These phenomena cannot happen in deterministic dynamical systems of a single map $f \in \rat+$, 
since,  in deterministic dynamical system of a rational map $f \in \rat+$, it is well known that the following holds for the Julia set $J(f)$ of $f$. 
For every neighborhood $O$ of a point of the Julia set $J(f),$  there exists $N \in \nn$ such that $f^{\circ n} (O) \supset J(f)$ for every $n \geq N.$ 
Here,  $f^{\circ n} $ denotes the $n$-th iterate of $f$. 
Note that $\diam J(f) > 0. $
We refer the readers to Milnor's book \cite{miln06} or Carleson and Gamelin's book \cite{CG} for details.
Besides,  it was shown by Ma\~{n}\'e \cite{mane} that 
the set of all points $z \in \rs$ with  $\liminf_{n \to \infty} n^{-1} \log \|D  f^{\circ n} (z) \|> 0$ has positive Hausdorff dimension.  
In this sense,  Main Result \ref{mr:contractingBall} and Main Result \ref{mr:NegLyap}  describe noise-induced order. 

We next consider how many mean stable systems exist. 
For our purpose, we give the following definition. 
 
\begin{definition}
Let $X \subset \rat+$.  
Define  $\mathcal{A}(X)$ as the set of all $\tau \in \MRDS(X)$ which are mean stable.  
Further, define  $\mathcal{C}(X)$ as the set of all $\tau \in \MRDS(X)$ which satisfy 
$\mathbb{J}(S_\tau) = \rs \times V$ and $\overline{\bigcup_{h \in H_i^j(S_\tau)}\{h(z)\}} = \rs$ for each $i,j \in V$ and $z\in \rs$. 
\end{definition}

Note that for each  $\tau \in \mathcal{C}(X)$, the set-valued dynamics of $S_{\tau}$ is topologically chaotic, which also   describes a noise-induced phenomenon.  
See Example \ref{ex:C}. 
We present the following  results regarding the thickness of $\mathcal{A}$ and $\mathcal{C}$. 

\begin{result}[Corollary \ref{cor:AisOpen} and Theorem \ref{th:denseMainRes}]\label{mr:dense}
Let $X \subset \rat+$. 
Then the set $\mathcal{A}(X)$ is  open in $\MRDS(X)$.
Moreover, if $X$ is $\condition$,  
then the disjoint union $\mathcal{A}(X) \cup \mathcal{C}(X)$ is dense in the space $\MRDS(X).$ 
\end{result}

As a corollary, we have the following result regarding the polynomial dynamics.

\begin{cor}[See Corollary \ref{cor:OD3}]\label{cor:OD}
Let $X$ be a non-degenerate subset of $\mathrm{Poly}$. Then the set $\mathcal{A}(X)$ is open and dense in $\mathrm{MRDS}(X)$. In particular, 
the set  $\mathcal{A}(\Poly)$  is open and dense in $\MRDS(\Poly).$ 
\end{cor}

Corollary \ref{cor:OD}  is related to the famous conjecture of Hyperbolic Density (HD conjecture). 
See McMullen's book \cite{McM94} for details. 
The dynamics of iteration of a hyperbolic rational map $f$ has the following properties, 
which is somewhat similar to the dynamics of a mean stable system. 
A hyperbolic rational map $f$ is expanding on the Julia set whose area is zero, 
and every orbit is contracted to an attracting cycle with negative Lyapunov exponent on the Fatou set. 
The set of all hyperbolic rational maps is conjectured to be open and dense in the space $\rat+$. 
We solved the RDS version of the HD conjecture in some sense. 

Main Result \ref{mr:dense} is also related to
the dichotomy result for the real quadratic family by Lyubich \cite{L02}, which says
``almost every real quadratic map is either regular or stochastic''. 
Here, regular means hyperbolic and stochastic means that the dynamics has an absolutely continuous invariant probability measure. 
The latter seems to be similar to the element $\tau$ of $\mathcal{C}$ in our context 
although we do not show that 
it has an absolutely continuous invariant measure. 
For real analytic families of unimodal maps, see also the papers  \cite{ALdM} by Avila, Lyubich and de Melo and \cite{AM} by Avila and Moreira. 
Compared to these results, we need to control randomness which is possibly large. 
This large noise causes difficulty for our analysis. 

Last but not least, we consider  families of $\MRDS$s.  
We show that  for such a family,  the set of mean stable systems has full measure under certain conditions.  

\begin{result}[Theorem \ref{th:bif}]\label{mr:biffam}
Suppose $X \subset \rat+$ is $\condition$.  
Let $\Lambda$ be a topological space and let $m$ be a $\sigma$-finite Borel measure on $\Lambda$. 
Let $I =[ a, b )$ be an  interval on the real line $\mathbb{R}$, possibly $I= [a, \infty )$. 
Suppose $\Phi \colon \Lambda \times I \to \MRDS(X)$ satisfies the following three conditions. Denote $ \Phi (\lambda, s) =  \tau^{\lambda, s}$. 
 
\begin{enumerate}
\item $\Phi$ is continuous and the associated directed graphs $(V, E)$ of $\tau^{\lambda, s}$ are identical for all $(\lambda, s) \in \Lambda \times I$. 
\item $\supp \tau_{e}^{\lambda, s_{1}} \subset \mathrm{int} (\supp \tau_{e}^{\lambda, s_{2}})$ for each $e \in E,$ $\lambda \in \Lambda$ and $s_{1} < s_{2}$, where $\mathrm{int}$ denotes the set of all interior points with respect to the topological space $X.$ 
\item $ \tau^{\lambda, s}$ has at least one attracting minimal set for each $(\lambda, s) \in \Lambda \times I.$ 
\end{enumerate}

For each $s \in I,$ we  denote by $\mathrm{Bif}_s$ the set of all  $\lambda \in \Lambda$ satisfying  that $\tau^{\lambda, s}$ are not mean stable. 
Also, for each $\lambda \in \Lambda,$ we denote by $\mathrm{Bif}^{\lambda}$ the set of all $s \in I$ such that $\tau^{\lambda, s}$ are not mean stable. 
Suppose that there exists $\alpha  \in \nn$ such that $\# \mathrm{Bif}^{\lambda} \leq \alpha$   for each $\lambda \in \Lambda.$  

Then $m ( \mathrm{Bif}_s ) =0$ for all but countably many $s \in I.$    
\end{result}

We remark the following corollary holds. 

\begin{cor}[Corollary \ref{cor:7}]\label{cor:Bif}
Let $X, \Lambda, m, I, \Phi$ as in Setting \ref{setting:biffam}. 
Suppose that there exists $d \in \mathbb{N}$ with $d \geq 2$ such that $2 \leq \deg (g) \leq d$ for each $g \in X$. 
Then there exists $\alpha  \in \nn$ such that $\# \mathrm{Bif}^{\lambda} \leq \alpha$   for each $\lambda \in \Lambda.$
Hence, $m ( \mathrm{Bif}_s ) =0$ for all but countably many $s \in I.$   
\end{cor}

It is interesting that our result can be applied to the quadratic family $f_{c}(z)= z^{2} +c$, see Example \ref{ex:quad}. 
On the dynamics of iteration of a single map,  
Shishikura \cite{Shi} showed that the bifurcation locus of the quadratic family, namely the boundary of the Mandelbrot set, has Hausdorff dimension $2$. 
However, it is still open whether or not the boundary of the Mandelbrot set has positive area, see \cite{MNTU}. 
We solved the RDS version of this problem in a general form. 

Note that Main Result \ref{mr:biffam}  is a new result even for i.i.d.\ systems. 
If $m=\# V =1$, then Main Result \ref{mr:contractingBall},  Main Result \ref{mr:NegLyap} and  Main Result \ref{mr:dense} coincide with the results for i.i.d.\ systems, which were shown in \cite{sumi13}. 
Also, our definition of mean stability coincides with the definition for i.i.d.\ systems if $m=1$. 
Our results and concepts are new for the case $m \neq 1$. 

The authors believe that we can work on ergodic properties of mean stable systems in the future.  
For mean stable i.i.d.\ systems, the first author proved in \cite{sumi11} that 
there exist finitely many invariant probability measures (or cycles of probability measures) supported on minimal sets. 
Moreover, he showed the spectral decomposition for the (dual of) transition operators, the spectral gaps for the transition operators and other measure-theoretic results. 
The authors believe that we can generalize these results to non-i.i.d.\ settings.

\subsection{Structure of the paper}
In Section \ref{2}, we define minimal sets of MRDS and give the classification of them. 
More precisely, a minimal set is one of the three types; it intersects the Julia set, it intersects a rotation domain, or it is \textit{attracting} as defined in Section \ref{2}. 
This is the key to our work. 
In Section \ref{3}, we show the fundamental properties of mean stable systems. 
In particular, we explain the relation between mean stable systems and attracting minimal sets. 
By using these results,  we prove Main Results \ref{mr:contractingBall},  \ref{mr:NegLyap} and \ref{mr:dense}. 
In Section \ref{4}, we consider families of $\MRDS$ and investigate their bifurcations. 
Furthermore, we show Main Result \ref{mr:biffam}.

\subsection*{Acknowledgment}
The authors would like to thank Rich Stankewitz for his valuable comments. 
They are grateful to the referees for many detailed and constructive comments. 
The first author is partially supported by JSPS Grant-in-Aid for Scientific Research (B) Grant Number JP 19H01790. 
The second author is partially supported by  JSPS Grant-in-Aid for JSPS Fellows Grant Number JP 19J11045.

\section{Classification of minimal sets of Markov RDS}\label{2}
In this section, we consider general graph directed Markov systems consisting of rational maps before analyzing Markov random dynamical systems. 
In \cite{sw19}, the authors defined more general GDMS regarding continuous self-maps, but we are concerned with rational maps of degree two or more in this paper. 

\begin{definition}
Let $(V,E)$ be a directed graph with finite vertices and finite edges, and let $\Gamma_e$ be a non-empty subset of $\rat+$ indexed by a directed edge $e \in E$.
We call $S = (V, E,(\Gamma_e)_{e \in E})$ a {\it graph directed Markov system} (GDMS for short).  
The symbol $i(e)$ (resp. $t(e)$) denotes the initial (resp. terminal) vertex of each directed edge $e \in E$.   
We say that $S = (V, E,(\Gamma_e)_{e \in E})$ is {\it compactly generated} if  $\Gamma_{e}$ is compact for each $e \in E$. 
\end{definition}

\begin{definition}
We say that a GDMS   $S = (V, E,(\Gamma_e)_{e \in E})$ is {\it irreducible} if the directed graph $(V, E)$ is strongly connected. 
We say that MRDS $\tau = (\tau_{ij})_{i, j =1, \dots, m}$ is {\it irreducible} if the induced GDMS $S_{\tau}$ is irreducible. 
\end{definition}

In this paper, we   usually assume that all GDMSs $S$  and MRDSs $\tau$ are irreducible. 
For each GDMS $S= (V, E,(\Gamma_e)_{e \in E})$, 
we define admissible words, $H^{j}_{i}(S)$, $F_i(S)$, $J_i(S)$, $\mathbb{F}(S)$ and $\mathbb{J}(S)$ 
similarly as in Definition \ref{def:FJ} replacing $S_{\tau}$ by $S$. 
For instance, the Fatou set $F_i(S)$ at  $i \in V$ is the set of  all points $z \in \rs$ for which there exists a neighborhood $U$ of $z$ in $\rs$ such that  the family  $\bigcup_{j\in V}H_i^j(S)$ of maps on $\rs$ is  equicontinuous on $U$, and  
the complement $J_i(S) := \rs \setminus F_i(S)$ is  the Julia set of $S$ at the vertex $i$.

If a GDMS $S= (V, E,(\Gamma_e)_{e \in E})$ is irreducible, then the Julia set $J_{i}(S)$ contains uncountably many points, see \cite[Lemma 4.8]{sw19}.  
In particular, the Fatou set $F_{i}(S)$ admits the hyperbolic  metric for each $i \in V$.

In order to analyze dynamical systems, it is useful to investigate minimal sets. 
For our purpose, we define minimal sets of GDMSs as follows. 

\begin{notation}\label{not:set-valued}
For a family $\Omega \subset \rat+$ and a set $Y \subset \rs$, we set  
$$\Omega (Y) := \bigcup_{f \in \Omega} f(Y),\,\Omega ^{-1}(Y) :=\bigcup_{f \in \Omega} f^{-1}(Y) .$$
If $ \Omega = \emptyset$, then we set $\Omega (Y) := \emptyset, \,
\Omega ^{-1}(Y) := \emptyset .$
\end{notation}

\begin{definition}
Let $S= (V, E,(\Gamma_e)_{e \in E})$ be an irreducible GDMS and 
let $K_{i}$ and $L_i$ be subsets of $\rs$ for each $i \in V$. 
We consider the families  $(K_i)_{i \in V}$ and $(L_i)_{i \in V}$  indexed by $i \in V$.
\begin{enumerate}
\item We say that $(L_i)_{i \in V}$ is {\it forward $S$-invariant} if $\Gamma_e( L_{i(e)}) \subset L_{t(e)}$  for all $e \in E$. 
\item We write $ (K_i)_{i \in V} \subset (L_i)_{i \in V}$ if $K_i \subset L_i$ for each $i \in V.$ 
\item  We say that  $(K_i)_{i \in V}$ is  a  {\it minimal}  set of $S$ if it is a minimal element of the set of all 
 $(L_i)_{i \in V}$ satisfying that  $L_{i}$ is non-empty and  compact for each  $i \in V$  and  $(L_i)_{i \in V}$   is forward $S$-invariant,  
with respect to the order $\subset .$
\end{enumerate}
Also, for any irreducible MRDS $\tau$, any minimal set of $S_{\tau}$ is called a minimal set of $\tau$.
\end{definition}

For  an irreducible GDMS $S= (V, E,(\Gamma_e)_{e \in E}),$ the Fatou set $(F_{i} (S) )_{i \in V}$ is   forward $S$-invariant. 
For the proof, see \cite[Lemma 2.15]{sw19}. 
The following can be proved easily. 

\begin{lemma}\label{lem:OrbitIsDence}
Let $S= (V, E,(\Gamma_e)_{e \in E})$ be an irreducible GDMS and 
let $(L_i)_{i \in V}$ be a minimal set of $S$. 
Then, for each $k \in V$ and each $z_0 \in L_k$, we have $L_i = \overline{  H_k^i(S) (\{ z_0\}) } $ for each $i \in V$. 
\end{lemma}

\begin{proof}
Define $K_i =  \overline{  H_k^i(S) (\{ z_0\}) }  .$  
It is easy to prove that $(K_i)_{i \in V}$ is forward $S$-invariant. 
Since $(L_i)_{i \in V}$ is forward $S$-invariant and $z_0 \in L_k$, we have $K_i \subset L_i$ for each $i \in V$. 
Thus,  $L_i = \overline{  H_k^i(S) (\{ z_0\}) }  $ for each $i \in V$  by the minimality of  $(L_i)_{i \in V}$. 
\end{proof}

\begin{lemma}\label{lem:selfsimilar}
Let $S= (V, E,(\Gamma_e)_{e \in E})$ be a compactly generated irreducible GDMS and 
let $(L_i)_{i \in V}$ be a minimal set. 
Then $L_j =  \bigcup_{t(e) = j} \Gamma_e( L_{i(e)}) $ for each $j \in V$.   
\end{lemma}

\begin{proof}
Note that $ \bigcup_{t(e) = j}  \Gamma_e( L_{i(e)}) $ is compact since $\Gamma_e$ is compact and $E$ is finite.  
Define $K_j =  \bigcup_{t(e) = j} \Gamma_e( L_{i(e)})   .$  
It is easy to prove that $(K_i)_{i \in V}$ is forward $S$-invariant. 
Since $(L_i)_{i \in V}$ is forward $S$-invariant, we have $L_j \supset  \bigcup_{t(e) = j}  \Gamma_e( L_{i(e)}).$ 
Thus,  $L_i =  \bigcup_{t(e) = j}  \Gamma_e( L_{i(e)})$ for each $i \in V$ by the minimality of  $(L_i)_{i \in V}$. 
\end{proof}

We now define attracting minimal set, 
which is one of the most important concepts in this paper. 

\begin{definition}\label{def:attractingMinimal}
Let $S= (V, E,(\Gamma_e)_{e \in E})$ be an irreducible GDMS and let $(L_i)_{i \in V}$ be a minimal set of $S.$ 
We say that  $(L_i)_{i \in V}$ is  an \textit{attracting} minimal set of $S$ if there exist $N \in \nn$ and open sets $(U_{i})_{i \in V}$ and $(W_{i})_{i \in V}$ such that 
\begin{enumerate}
\item $L_{i} \subset  U_{i} \subset \overline{U_{i}} \subset W_{i} \subset \overline{W_{i}} \subset F_{i}(S)$ for each $i \in V$ and 
\item  for each admissible word $e=(e_{1}, \dots, e_{N})$ with length $N$ and each $f_{n} \in \Gamma_{e_{n}} \, (n=1,\dots, N)$, we have $\overline{f_{N} \circ \dots  \circ f_{1} (W_{i(e)})} \subset U_{t(e)}.$  
\end{enumerate}
Also, for any irreducible MRDS $\tau$, any attracting minimal set of $S_{\tau}$ is called an attracting  minimal set of $\tau$.
\end{definition}

Attracting minimal sets play a crucial role to figure out the stability of (random) dynamical systems. 
Regarding minimal sets, we have some equivalent conditions for them to be attracting. 

\begin{lemma}\label{lem:attractingMinimal}
Let $S= (V, E,(\Gamma_e)_{e \in E})$ be a  compactly generated irreducible GDMS and let $(L_i)_{i \in V}$ be a minimal set for $S$ such that $(L_i)_{i \in V} \subset (F_i(S))_{i \in V} .$ 
Let $O_{i}$ be the finite union of the connected components of $F_{i}(S)$ each of which intersects $L_{i}.$ 
Denote by $\hypd$  the hyperbolic metric on each connected component of $O_{i}$  for each $i \in V.$ 
Then the following are equivalent.  
\begin{enumerate}
\item $(L_i)_{i \in V}$ is attracting. 
\item There exists $N \in \nn$ such that for each admissible word $e=(e_{1}, \dots, e_{N})$ with length $N$ and for each $f_{n} \in \Gamma_{e_{n}} \, (n=1,\dots, N)$, there exists $c \in (0, 1)$ such that for each connected component $U$ of $O_{i}$ and for each $x, y \in U,$ we have  
$$\hypd( f_{N}\circ \dots \circ f_{1} (x) ,  f_{N}\circ \dots \circ f_{1} (y)   ) \leq c  \hypd (x , y).$$
\item The constant $c$ above can be chosen so that $c$ does not depend on neither admissible words $e$ nor maps $f_{n}$; there exist $N \in \nn$ and  $c' \in (0, 1)$  such that for each admissible word $e=(e_{1}, \dots, e_{N})$ with length $N,$ for each $f_{n} \in \Gamma_{e_{n}} \, (n=1,\dots, N)$, for  each connected component $U$ of $O_{i}$ and for each $x, y \in U,$ we have  
$$\hypd( f_{N}\circ \dots \circ f_{1} (x) ,  f_{N}\circ \dots \circ f_{1} (y)   ) \leq c'  \hypd (x , y).$$
\end{enumerate} 
\end{lemma}

\begin{proof}
Statement (iii) immediately implies statement (i), and  statement (ii) implies statement (iii)  since $S$ is compactly generated. 

Suppose that $(L_{i})_{i \in V}$ is an attracting minimal set, and we show  the statement (ii) holds.  
Take $N \in \nn$, and open sets $(U_{i})_{i \in V}$ and $(W_{i})_{i \in V}$ as in Definition  \ref{def:attractingMinimal}. 
For each admissible word $e=(e_{1}, \dots, e_{N})$ with length $N$ and for each $f_{n} \in \Gamma_{e_{n}} \, (n=1,\dots, N)$, 
 we have $\overline{f_{N} \circ \dots  \circ f_{1} (W_{i(e)})} \subset U_{t(e)} \subset W_{t(e)}.$  
 Also, $f_{N}\circ \dots \circ f_{1}( O_{i(e)} ) \subset O_{t(e)}$ since $(F_{i} (S) )_{i \in V}$ is   forward $S$-invariant. 
It follows that there exists $c \in (0, 1)$ such that  for each connected component $U$ of $O_{i}$ and for each $x, y \in U,$ we have  
$$\hypd( f_{N}\circ \dots \circ f_{1} (x) ,  f_{N}\circ \dots \circ f_{1} (y)   ) \leq c  \hypd (x , y).$$
Thus, we have completed our proof. 
\end{proof}

The following proposition is very important to prove our main results. 
On (random) dynamics of holomorphic maps, we can classify the minimal sets as follows. 

\begin{prop}\label{prop:clsfcationOfMin}
Let $S= (V, E,(\Gamma_e)_{e \in E})$ be a compactly generated irreducible GDMS and let $(L_i)_{i \in V}$ be a minimal set of $S.$
Then $(L_i)_{i \in V}$ satisfies  one of the following three conditions. 
\begin{enumerate}
\renewcommand{\labelenumi}{(\Roman{enumi})}
\item The set  $(L_i)_{i \in V}$ intersects the Julia set; $L_{i} \cap J_{i}(S) \neq \emptyset$ for some $i \in V.$
\item The set $(L_i)_{i \in V} \subset (F_i(S))_{i \in V}$ and $(L_i)_{i \in V}$ intersects  a rotation domain; there exist $i \in V$ and $h \in H_{i}^{i}(S)$ such that $L_{i} \cap D \neq \emptyset$, where $D$ is  a connected component of $F_{i}(S)$ on which  $h$ is holomorphically conjugate to an irrational rotation on the unit disk or an annulus.  
\item The set $(L_i)_{i \in V}$ is attracting. 
\end{enumerate}
\end{prop}

\begin{proof}
Suppose that a minimal set $(L_i)_{i \in V}$ is neither of type (I) nor (II), and we show that $(L_i)_{i \in V}$ is of type (III). 
By our assumption, $L_{i}$ is contained in the Fatou set $F_{i}(S)$ at $i$ for each $i \in V.$ 
Let $O_{i}$ be the finite union of the connected components of $F_{i}(S)$ each of which intersects $L_{i}$ and  
denote by $\hypd$  the hyperbolic metric on each connected component of  $O_{i}$  for each $i \in V.$ 

We show that statement (ii) of Lemma \ref{lem:attractingMinimal} holds. 
Take a sufficiently large natural number  $N$, say the product of $1 + \# V,$ $\# V$ and $1 + N_{0}$, 
where $$N_{0} = \max_{i \in V} \{ \text{number of the connected components of } O_{i} \}.$$  
Then, for each admissible word $e=(e_{1}, \dots, e_{N})$ with length $N$ and for each $f_{n} \in \Gamma_{e_{n}} \, (n=1,\dots, N),$ there exist $1 \leq m \leq n \leq N$ and a connected component $A$ of $O_{i}$ for some $i \in V$  such that $i(e_{m}) = t(e_{n}) = i $ and $f_{n} \circ \dots \circ f_{m} (A) \subset A.$ 

Note that 
dynamics of $f_{n} \circ \dots \circ f_{m} $ on $A$ is well understood and classified as in \cite[\S 5]{miln06}. 
Since $L_{i} \cap A \neq \emptyset$ and   $(L_i)_{i \in V}$ is not of type (II),  the map $f_{n} \circ \dots \circ f_{m} $ is attracting so that there exists $c<1$ such that $\hypd (f_{n} \circ \dots \circ f_{m} (x) , f_{n} \circ \dots \circ f_{m}  (y)  ) \leq c \hypd (x, y)$ for each $x$, $y \in A.$ 
Thus,  for each connected component $U$ of $O_{i}$ and for each $x, y \in U,$ we have  
$\hypd( f_{N}\circ \dots \circ f_{1} (x) ,  f_{N}\circ \dots \circ f_{1} (y)   ) \leq c  \hypd (x , y),$ and this completes our proof.   
\end{proof}

\begin{definition}
Let $S= (V, E,(\Gamma_e)_{e \in E})$ be a compactly generated irreducible GDMS and let $\mathbb{L} =(L_i)_{i \in V}$ be a minimal set for $S.$
We say that $\mathbb{L} $ is \textit{J-touching} if $\mathbb{L}$ is of type (I) in Proposition \ref{prop:clsfcationOfMin}. 
We say that $\mathbb{L}$ is \textit{sub-rotative} if $\mathbb{L}$ is of type (II) in Proposition \ref{prop:clsfcationOfMin}. 
\end{definition}

In the following, we present examples which possess a J-touching or a sub-rotative minimal set respectively. 

\begin{example}\label{ex:Jtouching}
This example is due to \cite{bbr}. 
Let $V$ and $E$ be singletons. 
Define $f_{c}(z)= z^{2} +c$ for $c \in \CC$. 
Let $\Gamma = \{ f_{c} \in \poly ; | c | \leq 1/4 \}$. 
Then GDMS $(V, E, \Gamma)$ has a minimal set $\overline{D}=\{z \in \CC; | z | \leq 1/2 \}$ and 
its Julia set $J$ contains the Julia set of iteration of $f_{1/4}$, which implies $J \ni 1/2$. 
Hence, the minimal set $\overline{D}$ is J-touching at $z =1/2$. 
\end{example}

\begin{example}\label{ex:rotating}
Let $V$ and $E$ be singletons. 
Let $f(z)= \nu z + z^{2}$ be a polynomial map which is linearizable at $z=0$, say $\nu = \exp(2\pi i \xi)$ with the golden ratio $\xi$ \cite[\S11]{miln06}.  
Let $g \in \poly$ be a map which has an attracting fixed point at $z=0$, say $g(z) = z^{2}$. 
Define $\Gamma = \{ f, g \}$. 
Then  the GDMS $S = (V, E, \Gamma)$ has a minimal set $\{ 0 \}$, which is contained in the Fatou set. 
Since $f$ does not contract the hyperbolic metric near $z=0$, the minimal set  $\{ 0 \}$ is not attracting, and necessarily is sub-rotative. 
\end{example}

\section{The dichotomy of Markov RDSs}\label{3}
In this section, we discuss the property of mean stable systems and prove Main Results \ref{mr:contractingBall}, \ref{mr:NegLyap} and \ref{mr:dense}. 
We consider mean stable GDMSs $S$ as follows. 

\begin{definition}\label{def:meanstable3}
Let $S= (V, E,(\Gamma_e)_{e \in E})$ be an irreducible GDMS. 
We say that $S$ is \textit{mean stable} if there exist $N \in \nn$ and  two families of non-empty open sets $(U_{i})_{i \in V}$ and $(W_{i})_{i \in V}$ such that 
\begin{enumerate}
\renewcommand{\labelenumi}{(\Roman{enumi})}
\item $U_{i} \subset \overline{U_{i}} \subset W_{i} \subset \overline{W_{i}} \subset F_{i}(S)$ for each $i \in V,$ 
\item for each admissible word $e=(e_{1}, \dots, e_{N})$ with length $N$ and for each $f_{n} \in \Gamma_{e_{n}} \, (n=1,\dots, N)$, we have $\overline{f_{N} \circ \dots  \circ f_{1} (W_{i(e)})} \subset U_{t(e)},$ and   
\item for each $z \in \rs$ and $i \in V$, there exist $j \in V_{}$ and $h \in H_{i}^{j}(S)$ such that $h(z) \in W_{j}.$ 
\end{enumerate}
\end{definition}

\begin{example}\label{ex:MS}
Let $V$ and $E$ be singletons. 
Define $f_{c}(z)= z^{2} +c$ for $c \in \CC$. 
Let $\Gamma = \{ f_{c} \in \poly ; | c | \leq \epsilon \}$ for $\epsilon > 0$. 
Then it is easy to see that the GDMS $(V, E, \Gamma)$ is mean stable for sufficiently small $\epsilon$. 
In general, we can show that this GDMS is mean stable if $\epsilon \neq 1/4$. 
See Remark \ref{rem:BBR}. 
\end{example}

By Definition \ref{def:meanstable3}, 
if an MRDS $\tau$ is mean stable, then 
for every $i \in V$ and for every $z \in \rs$, there exist $j \in V$ and $h \in H_i^j(S_\tau)$ such that $h(z) \notin J_{j}(S_{\tau}).$ 
This property implies some interesting results, see \cite{sw19}. 

We now show some lemmas concerning relation with the mean stability and minimal sets. 

\begin{notation}
Let $X \subset \rat+$ and denote by $\Cpt(X)$ the space of all non-empty compact sets of $X.$ 
We endow $\Cpt(X)$ with the Hausdorff metric. 
\end{notation}

\begin{lemma}\label{lem:forwardInvMS}
Let $S= (V, E,(\Gamma_e)_{e \in E})$ be an irreducible GDMS which is mean stable. 
Then the open sets $(U_{i})_{i \in V}$ and $(W_{i})_{i \in V}$ for $S$ in Definition \ref{def:meanstable3} can be chosen such that the two  are  both forward $S$-invariant. 
\end{lemma} 

\begin{proof}
Take $N \in \nn$,  $(U_{i})_{i \in V}$ and $(W_{i})_{i \in V}$ as in Definition \ref{def:meanstable3}, which may not be forward $S$-invariant. 
For each $i \in V,$ define $U_i' = U_i \cup \bigcup \Gamma_{e_\ell} \circ \dots \circ \Gamma_{e_1} (U_{i(e_1)})$ where the union runs over all natural numbers $1 \leq \ell \leq N-1$ and all admissible words $(e_1, \dots, e_\ell) $ with length $\ell$ such that $t(e_\ell) = i$. 
Note that there are at most  finitely many numbers of such admissible words. 
Also, define $W_i' = W_i \cup \bigcup \Gamma_{e_\ell} \circ \dots \circ \Gamma_{e_1} (W_{i(e_1)})$ by a similar way. 
By the construction, $(U_{i}')_{i \in V}$ and $(W_{i}')_{i \in V}$ are both forward $S$-invariant. 

We show  $(U_{i}')_{i \in V}$ and $(W_{i}')_{i \in V}$ satisfy the conditions in Definition \ref{def:meanstable3}. 
It is trivial that conditions (I) and  (III) hold. 
We show  $\overline{f_{N} \circ \dots  \circ f_{1} (W_{i(e)}')} \subset U_{t(e)}'$ for each admissible word $e=(e_{1}, \dots, e_{N})$ with length $N$ and each $f_{n} \in \Gamma_{e_{n}} \, (n=1,\dots, N)$.   
Fix $w \in W_{i(e)}'$. 
If $w \in W_{i(e)}$, then  $f_{N} \circ \dots  \circ f_{1} (w) \in U_{t(e)} \subset U_{t(e)}'$. 
If $w \notin W_{i(e)}$, then there exist  admissible word $(\epsilon_1, \dots, \epsilon_\ell) $, $g_j \in \Gamma_{\epsilon_j}$ for each $j=1, \dots, \ell$ and $z \in  W_{i(\epsilon_1)}$ such that $w = g_\ell \circ \dots \circ g_1 (z)$ and $t(\epsilon_\ell) = i$. 
Then  
$f_{N} \circ \dots  \circ f_{1} (w) = f_{N} \circ \dots \circ f_{N-\ell+1}  (f_{N-\ell}\circ \dots \circ f_{1} \circ g_\ell \circ \dots \circ g_1 (z))$. 
 The right hand side belongs to $f_{N} \circ \dots \circ f_{N-\ell+1} (U_{t(e_{N-\ell})})$, and hence to $U_{t(e)}'.$ 
Thus, the condition (II) holds for $(U_{i}')_{i \in V}$ and $(W_{i}')_{i \in V}$, 
and this completes the proof.  
\end{proof}

\begin{lemma}\label{lem:forwardInvAttr}
Let $S= (V, E,(\Gamma_e)_{e \in E})$ be an irreducible GDMS and let $(L_{i})_{i \in V}$ be an attracting minimal set of $S$. 
Then the open sets $(U_{i})_{i \in V}$ and $(W_{i})_{i \in V}$ in Definition \ref{def:attractingMinimal} can be chosen such that the two  are  both forward $S$-invariant. 
\end{lemma}

\begin{proof}
The statement can be proved by a similar argument as Lemma \ref{lem:forwardInvMS}. 
\end{proof}

\begin{lemma}\label{lem:attrminsetisfinite}
Let $S= (V, E,(\Gamma_e)_{e \in E})$ be an irreducible GDMS. 
Then the number of attracting minimal sets of $S$ is finite. 
More precisely, for each  $j \in V$ and $h \in H_{j}^{j}(S)$, the number of attracting minimal sets of $S$ is at most the number of attracting cycles of $h,$ 
and hence at most $2 \deg (h) -2.$ 
\end{lemma}

\begin{proof}
Fix $j \in V$ and $h \in H_{j}^{j}(S).$ 
For each attracting minimal set $(L_i)_{i \in V}$ of $S$, there exist $N \in \nn$ and  open set $W_{j}$ such that $L_{j} \subset W_{j}$ and $\overline{h^{\circ N} (W_{j})} \subset W_{j},$ where $h^{\circ N}$ denotes $N$-th iterate of $h.$ 
It follows that $h^{\circ N}$ has an attracting periodic point $a$ in $W_{j}$.  
For a point $z \in L_{j},$  the orbit $h^{\circ N \ell}(z)$ accumulates to $a$ as $\ell$ tends to infinity, and hence $a \in L_{j}.$ 
Thus,  the number of attracting minimal sets is at most  the number of attracting cycles of $h.$    
\end{proof}

\begin{lemma}\label{lem:allAttrMinThenMS}
Let $S= (V, E,(\Gamma_e)_{e \in E})$ be an  irreducible GDMS. 
If all  minimal sets of $S$ are attracting, then $S$ is mean stable. 
\end{lemma}

\begin{proof}
Suppose that all  minimal sets of $S$ are attracting. 
By Lemma \ref{lem:attrminsetisfinite},  the number of attracting minimal sets is finite. 
For each attracting minimal set, take a natural number and two kinds of open sets as in Definition \ref{def:attractingMinimal}.  
Define $N$ as the product of these natural numbers and define $(U_{i})_{i \in V}$ and $(W_{i})_{i \in V}$ as the union of these open sets respectively. 
Then it is easy to see that conditions (I) and (II) of Definition \ref{def:meanstable3} holds. 
We show that condition (III) also holds. 
For each $z \in \rs$ and $i \in V,$ define $K_{j}=  \overline{H_{i}^{j} (S)(\{z\})}$ for each $j \in V$.
Then $(K_{j})_{j \in V}$ is forward $S$-invariant, and it follows from Zorn's lemma that  there exists a minimal set $(L_{j})_{j \in V} $ such that $(L_{j})_{j \in V}  \subset (K_{j})_{j\in V} .$   
The minimal set $(L_{j})_{j \in V}$ is attracting by our assumption, thus there exists $h \in H_{i}^{j}(S)$ such that $h(z) \in W_{j}.$ 
This completes our proof.  
\end{proof}

\begin{lemma}\label{lem:ifMSthenminimalAreAttr}
Let $S= (V, E,(\Gamma_e)_{e \in E})$ be an  irreducible GDMS. 
If  $S$ is mean stable,  then all  minimal sets of $S$ are attracting. 
\end{lemma}

\begin{proof}
For a mean stable $S,$ take $N \in \nn$, $(U_{i})_{i \in V}$ and $(W_{i})_{i \in V}$  as in   Definition \ref{def:meanstable3}. 
We may assume that both  $(U_{i})_{i \in V}$ and $(W_{i})_{i \in V}$ are  forward $S$-invariant by Lemma \ref{lem:forwardInvMS}. 
Pick any minimal set $(L_i)_{i \in V}$. It suffices  to show $L_i \subset U_i$ for each $i \in V$. 
Fix $k \in V$ and $z \in L_k,$ then there exist $j \in V$ and $h \in H_k^j(S)$ such that $h(z) \in U_j.$ 
Since  $(L_i)_{i \in V}$ is forward invariant, we have  $z_0 = h(z) \in L_j$,  and hence $L_i = \overline{H_{j}^{i}(S) (\{z_0\})}$ for each $i \in V$ by Lemma \ref{lem:OrbitIsDence}. 
By Lemma \ref{lem:selfsimilar} and conditions (I) and (II) in Definition \ref{def:meanstable3} for $(U_{i})_{i \in V}$ and $(W_{i})_{i \in V}$,  
we can show $L_i \subset U_i$ for each $i \in V$. 
\end{proof}

By Lemmas \ref{lem:allAttrMinThenMS} and \ref{lem:ifMSthenminimalAreAttr}, we have the following corollary. 

\begin{cor}\label{cor:iff}
Let $S= (V, E,(\Gamma_e)_{e \in E})$ be an  irreducible GDMS. 
Then $S$ is mean stable if and only if all  minimal sets of $S$ are attracting. 
\end{cor}

By the corollary above, we can construct MRDSs which are not mean stable by using Examples \ref{ex:Jtouching} and \ref{ex:rotating}.

We now show that if $\tau$ is mean stable, then for every initial value,  sample-wise dynamics is contractive and the Lyapunov exponent is negative almost surely 
with respect to the natural measure $\tilde{\tau}$ associated with $\tau$.  
We first define the measure $\tilde{\tau}.$

\begin{definition}\label{titil_def}
For  ${\tau} \in \MRDS(\rat+)$, we set $p_{ij}=\tau_{ij}(\rat+)$ and set $P = (p_{ij})_{i,j \in V}$. 
Since $\tau$ is irreducible, there exists a unique vector $p=(p_1,\dots,  p_m)$  such that $pP=p$, $\sum_{i \in V} p_i =1$ and $p_i > 0$ for all $ i\in V$.

We define the Borel probability measure $\tilde{\tau}$ on $(\rat+ \times E)^{\nn}$ as follows. 
For each $i \in V,$ let $\tilde{\tau}_{i}$ be the Borel probability measure on $(\rat+ \times E)^{\nn}$ such that 
for any $N \in \nn$, for $N$ Borel sets $A_n \,(n=1,\dots,N)$ of  $\rat+$ and for  $(e_1, \dots , e_N) \in E^N$,
\begin{align*}
 &\tilde{\tau}_i \left( A'_1 \times \cdots \times A'_N \times \prod_{N+1}^{\infty} (\rat+ \times E) \right) \\
 =& \begin{cases}
     \tau_{e_1}(A_1) \cdots \tau_{e_N}(A_N), & \text{if } (e_1, \dots , e_N) \text{ is admissible and } i(e_{1})=i\\
     0, & \text{otherwise,}
   \end{cases}
\end{align*}
where $A'_n = A_n \times \{ e_n \} \,(n=1,\dots,N)$.
Then we define $\tilde{\tau}$ as the sum $\sum_{i \in V} p_{i} \tilde{\tau}_{i} $.
\end{definition}

We now prove Main Results \ref{mr:contractingBall} and \ref{mr:NegLyap}.  

\begin{theorem}[Main Result \ref{mr:contractingBall}]\label{th:ContraBall}
Let ${\tau} \in \MRDS(\rat+)$ be a mean stable system. 
Then we have all of the following. 
\begin{enumerate}
\item  There exists a constant $c \in (0,\, 1)$ satisfying that for each $z  \in \rs$, there exists a Borel subset $\mathfrak{F}$ of $(\rat+ \times E)^{\nn}$ with  $\tilde{\tau} (\mathfrak{F}) =1 $ 
 such that for every $(f_{n}, e_{n})_{n =1}^{\infty} \in \mathfrak{F}$, there exist $r= r(z, (f_{n}, e_{n})_{n =1}^{\infty}) >0$ and $K=K(z, (f_{n}, e_{n})_{n =1}^{\infty}) >0$ such that  $$ \diam  f_{n} \circ \dots  \circ f_{1} (B(z, r)) \leq K c^{n}$$ for every $n \in\nn.$   
\item  For each $z  \in \rs$, there exists a Borel subset $\mathfrak{F}$ of $(\rat+ \times E)^{\nn}$ with  $\tilde{\tau} (\mathfrak{F}) =1 $ 
 such that for every $(f_{n}, e_{n})_{n =1}^{\infty} \in \mathfrak{F}$, there exists an attracting minimal set $(L_{i})_{i \in V}$  such that  $d(f_{n}\circ \dots \circ f_1 (z), L_{t(e_{n})}) \to 0$ as $n \to \infty$. 
\item For every $i \in V$, the kernel Julia set 
$J_{\ker,i} (S_\tau) :=  \bigcap_{ j \in V} \bigcap_{h \in H_i^j(S) } h^{-1} (J_j(S))$ at $i$ is empty. 
\item There exists a Borel subset $\mathfrak{G}$ of $(\rat+ \times E)^{\nn}$ with  $\tilde{\tau} (\mathfrak{G}) =1 $ such that for every 
$\xi = (f_{n}, e_{n})_{n =1}^{\infty} \in \mathfrak{G}$, 
the ($2$-dimensional) Lebesgue measure of the Julia set $J_{\xi}$ is zero.   
Here, the Julia set $J_{\xi}$  of $\xi$  is the complement of  the set of  all points $z \in \rs$ for which there exists a neighborhood $U$ of $z$ in $\rs$ such that  the family $\{f_{n} \circ \dots  \circ f_{1}\}_{n=1}^{\infty}$ is  equicontinuous on $U$.  
\item There exist at most finitely many minimal sets of $S_{\tau}$. Moreover, 
each minimal set of $S_{\tau}$ is attracting. 
Also, for every minimal set $\mathbb{L} = (L_{i})_{i \in V}$, 
we define the function $\mathbb{T}_{\mathbb{L}} \colon \rs \times V \to [0, 1]$ by
$$\mathbb{T}_{\mathbb{L}}(z,i) := \tilde{\tau}_i (\{ \xi = (f_n , e_n )_{n \in \nn}  ; \, d(f_{n}\circ \dots \circ f_1 (z), L_{t(e_{n})}) \to 0 \, (n \to \infty) \} )$$ 
for every point  $(z,i) \in \rs \times V$. 
Then $\mathbb{T}_{\mathbb{L}}$ is continuous on $\rs \times V$.  
Also, the averaged function ${T}_{\mathbb{L}}(z) := \sum_{i \in V} p_{i} \mathbb{T}_{\mathbb{L}}(z,i)$ is continuous on $\rs,$ 
where  $(p_1,\dots,  p_m)$ is the probability vector in Definition \ref{titil_def}. 
\item Suppose $\tau$ has exactly $\ell$ minimal sets  $\mathbb{L}_{1}, \mathbb{L}_{2}, \dots, \mathbb{L}_{\ell}$. 
Then $${T}_{\mathbb{L}_{1}}(z) + {T}_{\mathbb{L}_{2}}(z) + \dots + {T}_{\mathbb{L}_{\ell}}(z) =1$$
for every $z \in \rs$. 
\end{enumerate}
\end{theorem}

\begin{proof}
The key to the proof of (i) and (ii) is  \cite[Lemma 3.10]{sw19}. 
For the mean stable $S_{\tau}$, fix a family of open sets $(W_{i})_{i \in V}$ and $N$ as in Definition \ref{def:meanstable3} and let $z  \in \rs.$ 
Then, by  \cite[Lemma 3.10]{sw19},  
$$\mathfrak{E} =  \{ (f_{n}, e_{n})_{n =1}^{\infty} \in (\rat+ \times E)^{\nn} \, ;  f_{n} \circ \dots  \circ f_{1} (z) \notin W_{t(e_{n})} \text{ for any } n\in \nn \} $$   
satisfies $\tilde{\tau}(\mathfrak{E}) =0,$ and we define $\mathfrak{F}$ as the complement of $\mathfrak{E}.$

For every $(f_{n}, e_{n})_{n =1}^{\infty} \in \mathfrak{F},$ there exists  $k \in \nn$ such that  $f_{k} \circ \dots  \circ f_{1} (z) \in W_{t(e_{k})}.$ 
Take $r= r(z, (f_{n}, e_{n})_{n =1}^{\infty}) >0$ such that $f_{k} \circ \dots  \circ f_{1} (B(z, r)) \subset W_{t(e_{k})}.$ 
Since dynamics on $(W_{i})_{i \in V}$ is uniformly contractive  with respect to the hyperbolic metric, there exists $c \in (0,\, 1)$, 
which does not depend on $z$, such that 
$$ \diam_{\mathrm{hyp}}  f_{\ell N +k} \circ \dots  \circ f_{1} (B(z, r)) \leq c^{\ell} \diam_{\mathrm{hyp}}  f_{k} \circ \dots  \circ f_{1} (B(z, r))$$ 
for every $\ell \in\nn.$ 
This completes the proof of (i) since hyperbolic metric and spherical metric are comparable on each compact set. 
By the similar method, we can show that the statement (ii) is also true. 

The statement (iii) is trivial by definition of the kernel Julia set and the mean stability, 
and it follows from \cite[Proposition 3.11]{sw19} that 
the $2$-dimensional Lebesgue measure of the Julia set $J_{\xi}$ is zero for $\tilde{\tau}$-almost every $\xi$.   
This completes the proof of (iv). 

We next show the statement (v) following \cite[Proposition 4.24]{sw19}. 
By Corollary \ref{cor:iff}, each minimal set of $S_{\tau}$ is attracting. 
By Lemma \ref{lem:attrminsetisfinite}, it follows that there exist at most finitely many minimal sets of $S_{\tau}$. 
For an attracting minimal set $\mathbb{L} = (L_{i})_{i \in V}$, for each $i \in V$ 
we let $O_{i}$ be the finite union of the connected components of $F_{i}(S_{\tau})$ each of which intersects $L_{i}.$ 
Then there exists a continuous function $\phi \colon \rs \times V \to [0, 1]$ such that   
$\phi(z, i) = 1$ for every $z \in L_{i}$ and $\phi(z, i) = 0$ for every $z \notin O_{i}$. 
For $\tau = (\tau_{ij} )_{i,j \in V}$, we define the transition operator $M_{\tau}$ of $\tau$ as follows.
$$M_{\tau} \psi (z,i) := \sum_{j \in V}  \int \psi (f (z),j) \,  {\rm d} \tau_{ij} (f) ,  \quad (z, i ) \in \rs \times V.$$
Then, it is easy to show that 
$\{ M_\tau^n \phi \}_{n \in \nn}$ converges pointwise to $ \mathbb{T}_{\mathbb{L}} $ on $\rs \times V$ as $n \to \infty$. 
Since the statement (ii) holds, the family $\{ M_\tau^n \phi \}_{n \in \nn}$ of continuous maps is equicontinuous on $\rs \times V$ by Proposition 3.11, Lemmas 2.39 and 2.38 of \cite{sw19}. 
It follows from the Arzel\`a-Ascoli theorem, the limit $ \mathbb{T}_{\mathbb{L}} $ is also continuous on $\rs \times V$. 
This completes the proof of (v). 

The statement (vi) is a direct consequence of (ii). 
\end{proof}

\begin{theorem}[Main Result \ref{mr:NegLyap}]\label{th:NegLyap}
Let ${\tau} \in \MRDS(\rat+)$ be a mean stable system. 
Then there exists $\alpha < 0$ such that  the following holds  for each $z  \in \rs$. 
There exists a Borel set $\mathfrak{F}$ with  $\tilde{\tau} (\mathfrak{F}) =1 $ such that for every $(f_{n}, e_{n})_{n =1}^{\infty} \in \mathfrak{F}$, we have 
$$\limsup_{n \to \infty} \frac{1}{n} \log \|D ( f_{n} \circ \dots  \circ f_{1}) (z) \| \leq \alpha.$$
Here, $Dg(z)$ denotes the complex differential of a map $g$ at $z$ and $\| \cdot \| $ denotes the norm with respect to the spherical metric. 
\end{theorem}

\begin{proof}
The statement can be proved by a similar argument as Theorem \ref{th:ContraBall}. 
\end{proof}

We now investigate perturbations of GDMSs and show that attracting minimal sets are stable under perturbations. 
For compactly generated GDMS $S = (V, E,(\Gamma_e)_{e \in E})$, we consider another compactly generated  GDMS $S'  = (V, E, (\Gamma_e')_{e \in E})$ such that $\Gamma'_{e}$ is close to $\Gamma_{e}$ with respect to  the Hausdorff metric  for each $e \in E.$

\begin{lemma}\label{lem:prtrbAttrMinSet}
Let  $X \subset \rat+$. 
Let $S= (V, E,(\Gamma_e)_{e \in E})$ be an  irreducible GDMS such that $\Gamma_{e} \in \Cpt(X)$ for each $e \in E$.  
Suppose that $S$ has  an attracting minimal set $(L_i)_{i \in V}.$   
Take $N \in \nn,$ open sets  $(U_{i})_{i \in V}$ and $(W_{i})_{i \in V}$ for $(L_i)_{i \in V}$ as in Definition \ref{def:attractingMinimal}. 
Also, take an open set $G_{i}$, which is close to $L_{i},$ such that $L_{i} \subset G_{i} \subset \overline{G_{i}} \subset U_{i}$  for each $i \in V.$  

Then there exists an open neighborhood $\mathcal{U}_{e}$ of $\Gamma_{e}$ in $\Cpt(X)$  for each $e \in E$  such that for each compact set $\Gamma_{e}' \in \mathcal{U}_{e}$,  GDMS  $S' = (V, E, (\Gamma_e')_{e \in E})$ has a unique  minimal set $(L'_i)_{i \in V}$  such that $(L'_i)_{i \in V} \subset (G_{i})_{i \in V}$ and   $(L'_i)_{i \in V}$ is attracting for $S'.$ 
 \end{lemma}

\begin{proof}
We may assume that both $(U_{i})_{i \in V}$ and $(W_{i})_{i \in V}$ are   forward $S$-invariant by Lemma \ref{lem:forwardInvAttr}. 
By Lemma \ref{lem:attractingMinimal}, $d (f_{n} \circ \dots \circ f_{1}(x) , L_{t(e_{n})} ) \to 0$ as $n \to \infty$ for each infinite admissible word  $(e_{1}, e_{2},  \dots )$, $f_{n} \in \Gamma_{e_{n}}$ $(n=1, 2, \dots)$ and $x\in \overline{W_{i(e_{1)}}}$. 
Thus, taking $N$ safficiently large, we may assume $\overline{f_{N} \circ \dots  \circ f_{1} (W_{i(e)})} \subset G_{t(e)}$  for each admissible word $e=(e_{1}, \dots, e_{N})$ with length $N$ and each $f_{n} \in \Gamma_{e_{n}} \, (n=1,\dots, N)$. 

Then  there exists an open neighborhood $\mathcal{U}_{e}$ of $\Gamma_{e}$ for each $e \in E$ with respect to the topology of $\Cpt(X)$ such that for each $e \in E$ and $\Gamma_{e}' \in \mathcal{U}_{e}$, another GDMS  $S'= (V, E, (\Gamma_e')_{e \in E})$ satisfies the following two conditions. 
\begin{itemize}
\item[(a)] $\overline{g_{N} \circ \dots  \circ g_{1} (W_{i(e)})} \subset G_{t(e)}$ for each admissible word $e=(e_{1}, \dots, e_{N})$ with length $N$ and each $g_{n} \in \Gamma'_{e_{n}} \, (n=1,\dots, N)$. 
\item[(b)] $g_{\ell} \circ \dots  \circ g_{1} \left(\overline{U_{i(e_{1})}}\right) \subset W_{t(e_{\ell})}$ for each $1 \leq \ell \leq N-1$, admissible word $e=(e_{1}, \dots, e_{\ell})$ with length $\ell$ and  $g_{j} \in \Gamma'_{e_{j}} \, (j=1,\dots, \ell)$. 
\end{itemize}
Fix  $\Gamma_{e}' \in \mathcal{U}_{e}$ for each $e \in E,$ and  
define $S'= (V, E, (\Gamma_e')_{e \in E})$. 

For each $i \in V$, we consider 
${K}'_{i} = \overline{\bigcup_{k \in V} H_{k}^{i}(S') ({L}_{k})}.$ 
Then $({K}'_{i} )_{i\in V} $ is a family of compact sets which is  forward $S'$-invariant, 
and hence there exists a minimal set $({L'_{i}})_{i \in V}$ of $S'$ such that $({L'_{i}})_{i \in V} \subset ({{K}'_{i}})_{i \in V}.$ 
Also, it follows from (a) and (b) that $({K}'_{i} )_{i\in V} \subset ({W}_{i})_{i\in V} $,
and consequently $({L}'_{i} )_{i\in V} \subset ({W}_{i})_{i\in V}$. 
Using Lemma \ref{lem:selfsimilar} repeatedly, for each $i\in V$, we have  $L'_{i} = \bigcup \Gamma'_{e_{N}} \circ \dots \circ \Gamma'_{e_{1}}(L'_{i(e_{1)}})$ where the union runs over all admissible words $e=(e_{1}, \dots, e_{N})$ with length $N$ such that $t(e)=i$. 
It follows from (a) that $(L'_i)_{i \in V} \subset (G_{i})_{i \in V}.$ 
Also, by Montel's theorem, we have that  $W_{i} \subset F_{i}(S')$ for each $i \in V,$ and hence $({L'_{i}})_{i \in V}$  is an attracting minimal set of $S'.$ 
We now show the uniqueness of the minimal set in $(G_{i})_{i \in V}$. 
Since $G_{i}$ is close to $L_{i}$, we may assume that for every point $p \in G_{i}$, the $S'$-orbit tends to the attracting minimal set $(L'_{i})_{i \in V}$. 
Thus, minimal sets in $(G_{i})_{i \in V}$ is unique. 
This completes our proof. 
\end{proof}

Using the lemma above, we can control global stability as follows.

\begin{prop}\label{prop:prtbMS}
Let  $X \subset \rat+$. 
Suppose GDMS $S= (V, E,(\Gamma_e)_{e \in E})$  is irreducible and  mean stable with  $\Gamma_{e} \in \Cpt(X)$ for each $e \in E$. 
Then there exists an open neighborhood $\mathcal{U}_{e}$ of $\Gamma_{e}$ in $\Cpt(X)$  for each $e \in E$  such that 
for each $\Gamma_{e}' \in \mathcal{U}_{e}$, another GDMS  $S'= (V, E, (\Gamma_e')_{e \in E})$ is also mean stable. 
\end{prop}

\begin{proof}
By Lemma \ref{lem:ifMSthenminimalAreAttr} and Lemma \ref{lem:attrminsetisfinite}, 
$S$ has finitely many minimal sets, which are all attracting.  
It follows from Lemma \ref{lem:prtrbAttrMinSet} that  there exists an open neighborhood $\mathcal{U}_{e}$ of $\Gamma_{e}$  for each $e \in E$ with respect to the Hausdorff metric such that for each $\Gamma_{e}' \in \mathcal{U}_{e}$, 
another GDMS  $S'= (V, E, (\Gamma_e')_{e \in E})$ satisfies the following. 
For each attracting minimal set of $S$, another system $S'$ has the attracting  minimal set  close to it.   

By Lemma \ref{lem:allAttrMinThenMS}, it suffices to show that $S'$ does not have any minimal set except these attracting ones. 
For mean stable GDMS $S$, we take  $(W_{i})_{i \in V}$  as in   Definition \ref{def:meanstable3}. 
Then, for each $z \in \rs$ and $i \in V$, there exist $j \in V_{}$ and $h \in H_{i}^{j}(S)$ such that $h(z) \in W_{j}.$ 
Choose an open neighborhood $O_{z}$ of $z$ so that  $h\left(\overline{O_{z}}\right) \subset W_{j}.$ 
Since $\rs$ is compact, there exist finitely many $O_{z_{1}}, \dots, O_{z_{m}}$ and $h_{1}, \dots, h_{m} \in H_{i}^{j}(S)$ such that $\rs = \bigcup_{\ell=1}^{k} O_{z_{\ell}}$ and $h_{\ell}\left(\overline{O_{z_{\ell}}}\right) \subset W_{j}.$ 
Taking the open neighborhood $\mathcal{U}_{e}$ so small if necessary, 
we may assume that there exist  $h'_{1}, \dots, h'_{m} \in H_{i}^{j}(S')$ such that $h'_{\ell}\left(\overline{O_{z_{\ell}}}\right) \subset W_{j}.$ 
Therefore, for each $z \in \rs$ and $i \in V$, there exists $h' \in H_{i}^{j}(S')$  such that $h'(z) \in W_{j},$ and hence  $S'$ does not have J-touching nor sub-rotative minimal sets.  
This completes our proof. 
\end{proof}

The following is  a consequence of Proposition \ref{prop:prtbMS}. 

\begin{cor}[A part of Main Result \ref{mr:dense}]\label{cor:AisOpen}
Let $X \subset \rat+$. Then 
the set  $\mathcal{A}(X)$ of all mean stable systems is open in $\MRDS(X)$. 
\end{cor}

We now show the key lemma to Main Results \ref{mr:dense} and \ref{mr:biffam}. 
The assumption that $X$ is $\condition$ is essential here. 
It is easy to check that  $\rat+$, $\Poly$ and  $\{f(z)  + c\,  ; \,  c \in \mathbb{C} \}$ $ (f \in \Poly)$ is $\condition$ for example. 
Note that we see that any set $X$ in the following Example \ref{ex:degenerate} does not satisfy the $\condition$ condition. 

\begin{example}\label{ex:degenerate}
Define $g_{\theta}(z) = \theta z (1-z)$ for each $\theta \in \CC \setminus \{0\}$. 
Then $X=\{ g_{\theta} \in \poly ; \theta \in \CC \setminus \{0\} \}$ is NOT $\condition$ since they have a common fixed point $z=0$. 
Also, for $P \in \poly$, define $N_{\theta}(z) = z - \theta P(z)/P'(z)$ for each $\theta \in \CC \setminus \{0\}$. 
Then  $X=\{ N_{\theta} \in \rat+ ; \theta \in \CC \setminus \{0\}\}$ is NOT $\condition$ 
since all $N_{\theta}$ have common fixed points $z_{0}$ which satisfy $P(z_{0}) = 0$. 
However, we can consider their mean stability in the weak sense. See \cite{sumi21}. 
\end{example}

\begin{lemma}\label{lem:AisDense}
Suppose  $X \subset \rat+$ is  {\condition}. 
Let $\tau \in \MRDS(X).$  
If $S_{\tau} = (V, E,(\Gamma_e)_{e \in E})$ has an attracting minimal set, then $\tau \in \overline{\mathcal{A}(X)}$ where the closure is taken in the space $\MRDS(X)$ with respect to the topology in Definition \ref{def:topMRDS}.    
\end{lemma}

\begin{proof}
We approximate $\tau_{e}$ by a Borel measure $\rho_{e}$ on $X$ such that the total measure of $\tau_{e}$ and $\rho_{e}$ coincide   and  $\supp \tau_{e} \subset  \mathrm{int}  (\supp \rho_{e})$ for each $e \in E$, 
where $\mathrm{int}$ denotes the set of all interior points in the space $X$ endowed with the relative topology from $\rat+$. 
Also, define $\rho_{e} =0$ if $\tau_{e} =0.$
Note that the associated directed graph of $\rho$ is the same as that of $\tau$. 
Set  $\rho = (\rho_{e})_{e \in E} \in \MRDS(X).$  
We show that $\rho \in \mathcal{A}(X)$ if  $\supp \rho_{e}$ is sufficiently close to $\supp \tau_{e}$. 
Let $(L_i)_{i \in V}$ be an attracting minimal set of original system $S_{\tau}.$ 
By Lemma \ref{lem:prtrbAttrMinSet}, if $\supp \rho_{e}$ is sufficiently close to $\supp \tau_{e}$ for each $e \in E$, then  there exists an attracting minimal set  $(L'_i)_{i \in V}$ of $S_\rho.$ 
Pick any minimal set $(K'_i)_{i \in V}$ of $S_\rho.$ 
Since $\supp \tau_{e} \subset \supp \rho_{e}$ for each $e \in E,$ the set $(K'_i)_{i \in V}$ is forward $S_\tau$-invariant. 
By Zorn's lemma, there exists a minimal set $(K_i)_{i \in V}$ of $S_\tau$ such that $(K_i)_{i \in V} \subset  (K'_i)_{i \in V}.$  

We next  show  that $(K_i)_{i \in V}$ is attracting for $S_{\tau}$. 
Suppose that it is not attracting, then 
by Proposition \ref{prop:clsfcationOfMin}, there are possibly two cases. 

(I) For the J-touching case, there exists $i \in V$ and $p \in \rs$ such that $p \in K_i \cap J_i(S_\tau). $ 
By \cite[Proposition 2.16]{sw19}, we have  that $J_i(S_\tau) = \bigcup_{i(e) =i} \bigcup_{f \in \supp \tau_e} f^{-1} ( J_{t(e)} (S_\tau)).$ 
Thus, there exist $e \in E$ and $f  \in \supp \tau_e$ such that $f(p) \in  J_{t(e)} (S_\tau).$ 
Since $X$ is $\condition$, 
there exists a holomorphic family $\{ g_{\theta} \}_{\theta \in \Theta} $ of rational maps parametrized by a complex manifold $\Theta$ with $\{ g_{\theta} ; \theta \in \Theta\}  \subset X$ such that $ g_{\theta_{0}} = f$ for some $\theta_{0} \in\Theta$ and $\theta \mapsto g_{\theta}( p) $ is non-constant in any neighborhood of $\theta_{0}$. 
Taking $\Theta$ so small if necessary,  we may assume $\{ g_{\theta}; \theta \in \Theta\} \subset \mathrm{int} (\supp \rho_{e})$. 
Then $f(p) \in J_{t(e)} (S_\rho) \cap \mathrm{int} (K'_{t(e)})$ by the open mapping theorem.  
By \cite[Lemma 2.11]{sw19} and Montel's theorem, we have   that $H_{t(e)}^{t(e)} (S_\rho)$ does not omit three points on $ \mathrm{int} (K'_{t(e)}).$ 
However,  this contradicts the fact that $S_\rho$ has an attracting minimal set $(L'_i)_{i \in V}$   and another minimal set $(K'_i)_{i \in V}.$ 

(II) For the sub-rotative case, there exist $i \in V,$ $h \in H_{i}^{i}(S_{\tau})$ and $p \in \rs$ such that $p \in  K_{i} \cap D,$ where  $D$ is  a connected component of $F_i(S_\tau)$ on which $h$ is conjugate to an irrational rotation. 
Since $\supp \tau_{e} \subset  \mathrm{int}  \supp \rho_{e}$ for each $e \in E,$ there exists $g \in H_{i}^{i}(S_{\rho})$ such that $g(p) \in \partial D \subset J_i(S_\tau) \subset J_i(S_\rho).$ 
Then a similar  argument leads to a contradiction as in the case (I). 

Consequently,  $(K_i)_{i \in V}$ is an attracting minimal set of $S_\tau.$ 
Letting $\supp \rho_{e}$  sufficiently close to $\supp \tau_{e}$ for each $e \in E$, 
 we may assume that  $(K'_i)_{i \in V}$ is an attracting minimal set of $S_\rho$ by Lemma \ref{lem:prtrbAttrMinSet}.  
By Lemma \ref{lem:attrminsetisfinite}, $S_\tau$ has only finitely many attracting sets, and hence all the minimal sets of $S_\rho$ are attracting. 
Therefore, by Lemma \ref{lem:allAttrMinThenMS}, $S_\rho$  is mean stable. 
\end{proof}

As a corollary, we have the following result regarding the polynomial dynamics.

\begin{cor}[Corollary \ref{cor:OD}]\label{cor:OD3}
Suppose  $X \subset \Poly$ is $\condition$ and satisfies the condition (ii) in Definition \ref{def:condition}.  
Then the set $\mathcal{A}(X)$ of all mean stable polynomial dynamics is open and  dense in $\MRDS(X).$   
\end{cor}

\begin{proof}
By Corollary \ref{cor:AisOpen}, $\mathcal{A}(X)$ is open. 
Since $X \subset \Poly$, each $\tau \in \MRDS(X)$ has the attracting minimal set $(\{\infty \})_{i \in V}$. 
Thus, $\mathcal{A}(X)$ is dense by Lemma \ref{lem:AisDense}. 
\end{proof}

We now consider the complement of $\mathcal{A}(X)$. 

\begin{definition}\label{def:chaotic}
We say that a GDMS $S = (V, E,(\Gamma_e)_{e \in E})$ is \textit{\chaotic} if $\mathbb{J}(S) =\rs \times V$ and $(\rs)_{i\in V}$ is a minimal set of $S.$ 
We say that $\tau $ is $\chaotic$  if  associated GDMS $S_{\tau}$ is $\chaotic$, 
and define the set $\mathcal{C}(X)$ as the set of all $\tau \in \MRDS(X)$ which are $\chaotic$. 
\end{definition}

\begin{lemma}\label{lem:MinAndIntImplyC}
Let $X \subset \rat+$ and  $\tau \in \MRDS(X).$ 
If $(\rs)_{i \in V}$ is a minimal set of $S_\tau= (V, E,(\Gamma_e)_{e \in E})$ and $\mathrm{int} (J_j(S_\tau)) \neq \emptyset$ for some $j \in V,$ then $J_i(S_\tau) = \rs$ for each $i \in V,$ and hence $\tau \in \mathcal{C}(X).$    
\end{lemma}

\begin{proof}
If there exists $k \in V$ such that $J_k(S_\tau) \neq \rs, $ then  $J_i(S_\tau) \neq \rs $ for each $i \in V$ by the irreducibility of $\tau.$ 
Thus,  there exists a minimal set  $(L_i)_{i \in V}$ such that $L_j \subset \overline{F_j(S_\tau)}$ by Zorn's lemma. 
By the minimality of $(\rs)_{i \in V},$ we have $\rs = L_j.$ 
This contradicts the assumption that $\mathrm{int} J_j(S_\tau) \neq \emptyset.$ 
\end{proof}

\begin{example}\label{ex:C}
Let $f$ be a rational map whose Julia set is $\rs$, say a Latt\`es map \cite[\S 7]{miln06}. 
Let $m=1$ and define $\tau_{11} = \delta_{f}$, the Dirac measure at $f$. 
Then $\tau = (\tau_{11})$ is NOT $\chaotic$ since $f$ has periodic cycles as minimal sets. 
However, it still holds that there exists a dense (actually residual) set $R$ in $\rs$ such that for every $z \in R$, the orbit $\{z, f(z), f^{\circ 2}(z), \dots \}$ is dense in $\rs$, see \cite{miln06}. 

Using this map $f$, we can construct chaotic systems. 
We define a probability measure $\mu$ as the push-forward of the normalized Lebesgue measure under 
$$\{a \in \CC ; \frac{1}{2} \leq | a | \leq 1 \}  \ni a \mapsto f_{a}(z) = (a z^{2} + a) / (z^{2} + 1) \in \rat+$$
and define $\tau_{11} = \delta_{f}/2 + \mu/2.$
We now prove that $\tau = (\tau_{11})$ is $\chaotic$. 
It is enough to show that $\rs$ is minimal. 
For every $z \in \rs$, there exists $g \in \supp \mu$ such that $g(z) \in R$, where $R$ is the dense set above. 
Then $\{g(z), f(g(z)), f^{\circ 2}(g(z)), \dots \}$ is dense in $\rs$, hence $\rs$ is minimal. 

Also, let $m=2$ and define $\tau = (\tau_{ij})_{i, j=1,2}$ by
\begin{equation*}
\left(
    \begin{array}{cc}
      \tau_{11} & \tau_{12}  \\
      \tau_{21} & \tau_{22} \\
    \end{array}
  \right) = 
 \left(
    \begin{array}{cc}
      \frac{1}{2}\delta_{f} &  \frac{1}{2}\delta_{f} \\
      \mu & 0 \\  
    \end{array}       \right). 
\end{equation*}
Then we can show also that  $\tau = (\tau_{ij})_{i, j=1,2}$ is $\chaotic$. 
\end{example}

We now prove the density of mean stable or chaotic MRDSs, which is a part of Main Result \ref{mr:dense}. 

\begin{theorem}[Main Result \ref{mr:dense}]\label{th:denseMainRes}
Suppose  $X \subset \rat+$ is {\condition}. 
Then the union $\mathcal{A}(X) \cup \mathcal{C}(X)$ is dense in the space $\MRDS(X).$ 
\end{theorem}

\begin{proof}
Let $\tau \in \MRDS(X).$ 
If $S_\tau= (V, E,(\Gamma_e)_{e \in E})$ has an attracting minimal set, then $\tau \in \overline{ \mathcal{A}(X)}$  by Lemma \ref{lem:AisDense}. 
We assume  that $S_\tau$ has no attracting minimal sets. 
As in the proof of Lemma \ref{lem:AisDense},  we approximate $\tau_{e}$ by $\rho_{e}$ such that $\supp \tau_{e} \subset  \mathrm{int}  (\supp \rho_{e})$ for each $e \in E$ and define a new system $S_\rho.$  
Since $\emptyset \neq J_i(S_\tau) \subset J_i(S_\rho)$ for each $i \in V,$ we have   $\mathrm{int} (J_j(S_\rho)) \neq \emptyset$ for each $j \in V.$ 
Thus, by Lemma \ref{lem:MinAndIntImplyC}, it suffices to prove that $(\rs)_{i \in V}$ is a minimal set of $S_\rho.$ 

Take $z_0 \in \rs$ and $j \in V,$ and we  define $K'_i = \overline{ H_j^i(S_\rho) (\{ z_0\} )}$ for each $i \in V.$   
Since $( K'_i)_{i \in V}$ is forward $S_\rho$-invariant, there exist a minimal set $(L'_i)_{i \in V}$ of $S_\rho$ and a minimal set $(L_i)_{i \in V}$ of $S_\tau$ such that $( L_i)_{i \in V} \subset ( L'_i)_{i \in V} \subset  ( K'_i)_{i \in V}.$ 
Recall that $(L_i)_{i \in V}$  is not attracting for $S_{\tau}$ by our assumption, and hence there are two cases (I) or (II) by Proposition \ref{prop:clsfcationOfMin}. 

For the J-touching case (I), there exists $i \in V$  such that $L_i \cap J_i(S_\tau) \neq \emptyset. $
By a similar argument as the proof of Lemma \ref{lem:AisDense}, 
 there exists $e \in E$  such that $i(e) = i$ and $ J_{t(e)} (S_\rho) \cap \mathrm{int} (L'_{t(e)}) \neq \emptyset,$ and hence $H_{t(e)}^{t(e)} (S_\rho)$ does not omit three points on $ \mathrm{int} (L'_{t(e)}).$ 
It follows that     $(\rs)_{i \in V}$ is a minimal set of $S_\rho.$ 
For the sub-rotative case, there exist $i \in V$ and a rotation domain $D$ such that $ L_{i} \cap D \neq \emptyset.$  
The same idea  can be applied to show that $(\rs)_{i \in V}$ is a minimal set of $S_\rho.$ 
This completes our proof. 
\end{proof}


\section{Bifurcation of MRDSs}\label{4}
In this section, we consider  families of $\MRDS$s and their bifurcations.   

\begin{setting}\label{setting:biffam}
Let $\Lambda$ be a topological space and let $m$ be a $\sigma$-finite Borel measure on $\Lambda$, which we consider as a parameter space.
Let $I =[ a, b)$ be an  interval on the real line $\mathbb{R}$, possibly $I= [a, \infty )$. 
Suppose $X \subset \rat+$ is $\condition$ and 
suppose $\Phi \colon \Lambda \times I \to \MRDS(X)$ satisfies the following three conditions. 
Denote $ \Phi (\lambda, s) =  \tau^{\lambda, s}$. 
 
\begin{enumerate}
\item $\Phi$ is continuous and the associated directed graphs $(V, E)$ of $ \tau^{\lambda, s}$ are identical for all $(\lambda, s) \in \Lambda \times I$. 
\item $\supp \tau_{e}^{\lambda, s_{1}} \subset \mathrm{int} (\supp \tau_{e}^{\lambda, s_{2}})$ for each $e \in E,$ $\lambda \in \Lambda$ and $s_{1} < s_{2}$, where $\mathrm{int}$ denotes the set of all interior points with respect to $X$. 
\item $ \tau^{\lambda, s}$ has at least one attracting minimal set for each $(\lambda, s) \in \Lambda \times I.$ 
\end{enumerate}
\end{setting}

The essential assumption in Setting \ref{setting:biffam} is (ii), 
which describes the property that the size of noise increases as $s$ increases. 
This enables us to control minimal sets.

The most important example of such families is i.i.d.\ RDS of quadratic polynomial maps. 
The following is a motivating example, which itself is very interesting. 

\begin{example}\label{ex:quad}
Let $X=\{f_{c}(z)= z^{2} +c \in \poly ; c \in \CC\}$. 
Let $\Lambda = \mathbb{C}$ and identify $c \in \mathbb{C}$ with the quadratic polynomial $f_{c}$. 
Let $m=\Leb$ be the Lebesgue measure on $c$-plane and $I=[0, \infty).$ 
We define $ \Phi (c, s) =  \tau^{c, s}$ for $(c, s) \in \Lambda \times I$ as follows. 
Define $ \tau^{c, 0}$ as the Dirac measure at $c$ if $s=0,$ and otherwise 
define $ \tau^{c, s}$ as the Lebesgue measure on the disk $D(c, s)$ centered at $c$ with radius $s > 0$ normalized  by  $ \tau^{c, s}(D(c, s)) =1.$  
Regard it as an MRDS whose $V$ and $E$ are both singletons. 

Then  $(X,\, \Lambda,\, m,\, I,\, \Phi)$  satisfies the conditions (i), (ii) and (iii) in Setting \ref{setting:biffam}. 
\end{example}

\begin{example}
Suppose  $X \subset \Poly$ is $\condition$ and satisfies the condition (ii) in Definition \ref{def:condition} holds.  
Then the assumption (iii) in Setting  \ref{setting:biffam} is satisfied since for every $\tau \in \MRDS (X)$, $(\{\infty\})_{i \in V}$ is an attracting minimal set of $S_{\tau}$. 
\end{example}

We now show our results. 

\begin{lemma}
Let $X,\, \Lambda,\, m,\, I,\, \Phi$  as in Setting \ref{setting:biffam}. 
Then, for any $\lambda \in \Lambda,$  the number of minimal sets of $\tau^{\lambda, s}$ does not increase as $s$ increases. 
\end{lemma}

\begin{proof}
Fix $s_{1}\leq s_{2}.$ 
For a minimal set $(L'_{i})_{i \in V}$ of $\tau^{\lambda, s_{2}}$, there exists a minimal set $(L_{i})_{i \in V}$ of $\tau^{\lambda, s_{1}}$ such that $(L_{i})_{i \in V} \subset (L'_{i})_{i \in V}$ since $\supp \tau_{e}^{\lambda, s_{1}} \subset \supp \tau_{e}^{\lambda, s_{2}}$  for each $e \in E.$ 
Since minimal sets do not intersect one another, this completes our proof. 
\end{proof}

\begin{lemma}\label{lem:bifFinite}
Let $X,\, \Lambda,\, m,\, I,\, \Phi$ as in Setting \ref{setting:biffam} and 
suppose that there exists $d \in \mathbb{N}$ with $d \geq 2$ such that  $2 \leq \deg (g) \leq d$ for each $g \in X$.
Then there exists $\alpha  \in \nn$ such that  for each $\lambda \in \Lambda,$ the  number of $s \in I$ such that $\tau^{\lambda, s}$ is not mean stable is at most $\alpha.$ 
\end{lemma}

\begin{proof}
For the directed graph $(V,E)$, fix $i \in V.$ 
Then there exists an admissible word $e$ such that $i(e) =i= t(e),$ whose length is denoted by $N.$ 
Then, by our assumption,  there exists $h \in H_{i}^{i}(S_{\tau^{\lambda, a}})$ whose degree is at most $d^{N}.$ 
Define $\alpha = 2 d^{N} -2$ and take  $\lambda \in \Lambda.$ 
It follows from the proof of Lemma \ref{lem:AisDense} that for each $s_{0} \in I$, we have that  $ \tau^{\lambda, s}$ is mean stable for $s > s_{0}$ which is sufficiently close to $s_{0}.$  
Moreover, if  $ \tau^{\lambda, s_{0}}$ is not mean stable, 
then the number of minimal sets of $\tau^{\lambda, s}$ is strictly less than that of $\tau^{\lambda, s_{0}}$.  

If $s \in (a, b)$ is sufficiently close to $a,$ then $\tau^{\lambda, s}$ is mean stable. 
For mean stable $\tau^{\lambda, s},$  each minimal set  is attracting by Lemma \ref{lem:ifMSthenminimalAreAttr}.  
By Lemma \ref{lem:attrminsetisfinite}, the number of (attracting) minimal sets for $\tau^{\lambda, s}$ is at  most $\alpha.$ 
Since the number of minimal sets strictly decreases at $s$ where  $ \tau^{\lambda, s}$ is not mean stable, 
it follows that  the  number of $s \in I$ such that $\tau^{\lambda, s}$ is not mean stable is less than $\alpha.$ 
\end{proof}

We now prove Main Result \ref{mr:biffam}, which asserts the measure-theoretic thickness of mean stable MRDSs. 

\begin{theorem}[Main Result \ref{mr:biffam}]\label{th:bif}
Let $X,\, \Lambda,\, m,\, I,\, \Phi$ as in Setting \ref{setting:biffam} . 
Denote by $\mathrm{Bif}$ the set of all $(\lambda, s) \in \Lambda \times I$ satisfying that $\tau^{\lambda, s}$ is not mean stable. 
Besides, we define the sets $\mathrm{Bif}^{\lambda}$ and $\mathrm{Bif}_s$ as follows. 
\begin{enumerate}
\item[] For each $\lambda \in \Lambda$, 
we denote by $\mathrm{Bif}^{\lambda}$ the set of all  $s\in I$ satisfying that $\tau^{\lambda, s}$ is not mean stable. 
\item[] For each $s \in I$, 
we denote by $\mathrm{Bif}_s$ the set of all $\lambda \in \Lambda$ satisfying that  $\tau^{\lambda, s}$ is not mean stable. 
\end{enumerate}
Suppose that  there exists $\alpha  \in \nn$ such that $\# \mathrm{Bif}^{\lambda} \leq \alpha$   for each $\lambda \in \Lambda.$ 
Then $m ( \mathrm{Bif}_s ) =0$ for all but countably many $s \in I.$    
\end{theorem}

\begin{proof}
Since $(\Lambda, m)$ is $\sigma$-finite, we may assume that $m$ is a probability measure. 
We show by contradiction that $\# \{ s \in I \,; m( \mathrm{Bif}_s ) > n^{-1} \} \leq n \alpha.$  
Suppose that there exist mutually distinct elements  $s_1, \dots, s_{n \alpha +1} \in I$ such that $m( \mathrm{Bif}_{s_j} ) > n^{-1}$ for every $j =1, \dots, n \alpha +1.$ 
Then, 
\begin{align*}
\alpha + \frac{1}{n} = (n \alpha +1) \frac{1}{n} < \sum_{j=1}^{n \alpha +1} m( \mathrm{Bif}_{s_j}) 
= \int_\Lambda \#\{ s_j \in  \mathrm{Bif}^{\lambda} \, ; j  \in \{1, \dots, n \alpha +1\} \} \,  \mathrm{d} m (\lambda) \leq \alpha. 
\end{align*}
By contradiction, we have $\# \{ s \in I \,; m( \mathrm{Bif}_s ) > n^{-1} \} \leq n \alpha.$ 
We now let $I_0 = \bigcup_{n \in \nn} \{ s \in I \,; m( \mathrm{Bif}_s ) > n^{-1} \} $, which is countable. 
Then we have  $m ( \mathrm{Bif}_s ) =0$ for every $s \in I \setminus I_0.$ 
\end{proof}
 
 By Lemma \ref{lem:bifFinite} and Theorem \ref{th:bif}, we have the following corollary.

\begin{cor}[Corollary \ref{cor:Bif}]\label{cor:7}
Let $X,\, \Lambda,\, m,\, I,\, \Phi$ as in Setting \ref{setting:biffam}. 
Suppose that there exists $d \in \mathbb{N}$ with $d \geq 2$ such that  $2 \leq \deg (g) \leq d$ for each $g \in X$. 
Then  $m ( \mathrm{Bif}_s ) =0$ for all but countably many $s \in I.$   
\end{cor}

By Corollary \ref{cor:7}, we have the following corollary regarding the guadratic family. 

\begin{cor}
Let $X,\, \Lambda,\, m,\, I,\, \Phi$ as in Example \ref{ex:quad}. 
Then
we have $$\Leb (\{ c \in \CC\, ; \,  \tau^{c, s} \text{ is not mean stable} \}) =0$$ for all but countably many $s \in [0, \infty )$. 
\end{cor}

\begin{rem}\label{rem:BBR}
Br\"uck, B\"uger and Reitz in \cite{bbr} studied such i.i.d. RDSs. 
 They essentially showed that if the center $c$ satisfies $c = 0$, then the bifurcation occurs at
$s^{*} = 1/4.$ 
More precisely, $\tau^{0, s}$ has two minimal sets including $\{ \infty\}$ if $0 < s \leq s^{*}$, and 
 $\tau^{0, s}$ has the only one attracting minimal set $\{ \infty \}$ if $s > s^{*}$. 
 Hence,  $\tau^{0, s}$ is not mean stable if and only if $s = 0$ or $1/4$.
\end{rem}

\begin{rem}
Note also that $s^{*} = 1/4$ is the distance between $c = 0$ and the boundary of the celebrated Mandelbrot set. 
The Mandelbrot set $\mathcal{M}$ is the set of all parameters $c \in \CC$  such that the Julia set of $f_{c}(z) = z^{2} + c$ is connected. 
In general, it is easy to see that if  $f_{c}$ has an attracting cycle in $\CC$, then  there uniquely exists $s^{*} > 0$ such that   $\tau^{c, s}$ is not mean stable if and only if $s = 0$ or $s^{*}$. 
We need to investigate the bifurcation of our quadratic RDS thoroughly in the near future. 
\end{rem}


\end{document}